\documentclass[final,hidelinks]{siamart251216}

\usepackage{amssymb,mathtools}
\usepackage{mathrsfs}
\usepackage{bm}
\usepackage{bbm}
\usepackage{enumitem}
\usepackage{xcolor}

\newcommand{\R}{\mathbb{R}}

\newcommand{\E}{\mathbb{E}}

\newcommand{\K}{\mathcal{K}}

\newsiamthm{problem}{Problem formulation}
\newsiamremark{remark}{Remark}
\newsiamremark{example}{Example}

\theoremstyle{plain}
\theoremheaderfont{\scshape}
\theorembodyfont{\itshape}
\newtheorem{assumption}{Assumption}

\title{Ergodic Linear--Quadratic Optimal Control
with Stochastic Coefficients%
\thanks{\textbf{Funding:} This work is partially supported by
NSF grant DMS-2508581 and the Chair Finance and Sustainable Growth
of the Risk Foundation.}}

\author{
Xuyang Lin%
\thanks{New York University, Tandon School of Engineering
(\email{xuyang.lin@nyu.edu},
\email{nizar.touzi@nyu.edu}).}
\and
Nizar Touzi\footnotemark[2]
}

\headers{Ergodic LQ Control with Stochastic Coefficients}
{Xuyang Lin and Nizar Touzi}

\begin{document}

\maketitle

\begin{abstract}
We investigate ergodic linear--quadratic optimal control problems
with stochastic coefficients and introduce the associated stochastic
ergodic Hamilton--Jacobi--Bellman (HJB) equation, in which the
classical ergodic constant is replaced by a random process $\lambda$.
We establish well-posedness of the equation and show that $\lambda$
characterizes the optimal ergodic costs for a general class of
averaging kernels, including those corresponding to normalized
finite-horizon and discounted costs.
For random-periodic coefficients, we construct a random-periodic
solution and show that the optimal ergodic costs associated with
all admissible averaging kernels coincide.
We also establish stability of the optimal ergodic costs under
perturbations of the coefficients.
\end{abstract}



\section{Introduction}

Stochastic linear--quadratic (LQ) control problems form a classical class of
control problems, combining tractability with a wide
range of applications. The general form of their controlled dynamics can be written as
\begin{equation}\label{eq:lq-dynamics-intro}
dX_t^{x;\nu}
=
\left(A_tX_t^{x;\nu}+B_t\nu_t+f_t\right)dt
+\sum_{i=1}^d
\left(C_t^iX_t^{x;\nu}+D_t^i\nu_t+g_t^i\right)dW_t^i,\, X_0^{x;\nu}=x,
\end{equation}
where $W$ is a $d$-dimensional Brownian motion,
$\nu$ is a control, and
$A,B,C,D,f,g$ are given coefficients. For finite-horizon stochastic LQ problems,
the aim is to minimize the cost
$\mathbb E\int_0^T \ell_t(X_t^{x;\nu},\nu_t)\,dt,
$
where the running cost $\ell_t$ is quadratic in both
the state and the control, and the terminal cost is
taken to be zero for simplicity.

When the coefficients in \eqref{eq:lq-dynamics-intro}
and those of $\ell_t$ are stochastic,
Bismut \cite{Bismut1976} studied
a backward stochastic Riccati equation and an associated
dual BSDE:
\begin{equation}\label{eq:riccati-dual-intro}
dP_t=-\mathcal R_t(P_t,Q_t)\,dt+\sum_{i=1}^d Q_t^i\,dW_t^i,
\quad
dr_t=-\mathcal G_t(P_t,Q_t,r_t,k_t)\,dt+\sum_{i=1}^d k_t^i\,dW_t^i,
\end{equation}
on $[0,T]$ with suitable terminal conditions.
The minimization problem of the cost
$\mathbb E\int_0^T \ell_t(X_t^{x;\nu},\nu_t)\,dt
$ can be characterized
in terms of the solutions to \eqref{eq:riccati-dual-intro}
and the coefficients. Both equations fall outside the standard Lipschitz BSDE
framework, making them difficult to handle.
The general well-posedness problem remained open for many years
and was eventually resolved by Tang \cite{Tang2003} for the
Riccati equation and by Kohlmann and Tang \cite{KohlmannTang2003}
for the dual equation.

Beyond the finite-horizon setting, one may minimize the
infinite-horizon cost
$\mathbb E\int_0^\infty
\ell_t(X_t^{x;\nu},\nu_t)\,dt$,
or the ergodic costs defined by the limit inferior and
limit superior of
$J_T^{\rm av}(x;\nu)$ and $J_\rho^{\rm dis}(x;\nu)$
as $T\to\infty$ and $\rho\downarrow0$, respectively,
(see 
Arisawa and Lions~\cite{ArisawaLions1998},
and Bensoussan and Frehse~\cite{BensoussanFrehse1992}), where
\begin{equation}\label{eq:normalized-costs-intro}
J_T^{\rm av}(x;\nu)
:=\frac1T\mathbb E\int_0^T
\ell_t(X_t^{x;\nu},\nu_t)\,dt,
\quad
J_\rho^{\rm dis}(x;\nu)
:=\rho\mathbb E\int_0^\infty e^{-\rho t}
\ell_t(X_t^{x;\nu},\nu_t)\,dt.
\end{equation}
In the LQ setting with stochastic coefficients,
infinite-horizon and ergodic control problems can be studied
through
\eqref{eq:riccati-dual-intro} on $\R_+$.
Guatteri and Masiero \cite{GuatteriMasiero2010} considered
these problems under suitable stabilizability and coercivity
assumptions, with $g=0$.
They first studied \eqref{eq:riccati-dual-intro} on $\R_+$
and, under the condition
$f\in L^\infty\cap L^2(\Omega\times\R_+)$,
used its solutions to characterize
$\inf_\nu\mathbb E\int_0^\infty
\ell_t(X_t^{x;\nu},\nu_t)\,dt$.
For bounded $f$, they applied the preceding results to the
discounted problem and used a family of discount-dependent
Riccati and dual solutions to represent
$\liminf_{\rho\downarrow0}\inf_\nu J_\rho^{\rm dis}(x;\nu)$,
rather than the classical optimal ergodic cost
$\inf_\nu\liminf_{\rho\downarrow0}J_\rho^{\rm dis}(x;\nu)$. In contrast, as a special case of our characterization in
\eqref{eq:kernel-lower-value}--\eqref{eq:kernel-upper-value},
we establish the equality of these two quantities and express
their common value in terms of the solution of the undiscounted
equations, with the discount $\rho$ appearing
only in the explicit averaging weight $\rho e^{-\rho t}$.
 Ergodic LQ problems with stationary and random-periodic
coefficients were also studied by Guatteri and Masiero
\cite{GuatteriMasiero2009Stationary} and Wu and Zhang
\cite{WuZhang2026}, respectively.

A general approach to ergodic control in the
time-homogeneous Markovian setting, beyond the LQ case,
is based on the ergodic HJB equation
\cite{ArisawaLions1998,BensoussanFrehse1992}:
\[
\lambda=\inf_u\{\ell(x,u)+\mathscr L^u v(x)\},
\]
with $\mathscr L^u$ the controlled generator and
the constant $\lambda$ characterizing the optimal ergodic cost.

In the time-inhomogeneous Markovian LQ setting,
Da Prato and Ichikawa \cite{DaPratoIchikawa1990}
and Ankirchner and Engelhardt
\cite{AnkirchnerEngelhardt2022} characterized optimal
average costs through the Ces\`aro averages of an explicit
deterministic function $\lambda(t)$:
\begin{equation*}
\inf_\nu\limsup_{T\to\infty}J_T^{\rm av}(x;\nu)
=\limsup_{T\to\infty}\frac1T\int_0^T\lambda(t)\,dt.
\end{equation*}
In the latter work, $\lambda(t)$ appears in a
time-dependent HJB equation that can be rewritten as
\[
\lambda(t)=\partial_t v(t,x)
+\inf_u\{\ell_t(x,u)+\mathscr L_t^u v(t,x)\},
\]
where $v$ is constructed from solutions of the deterministic
Riccati and dual equations corresponding to
\eqref{eq:riccati-dual-intro} on $\R_+$.

In the present paper, building on the framework of \cite{GuatteriMasiero2010}, we study the ergodic
LQ control problem with general stochastic coefficients,
allowing $g\neq0$.
We first introduce a stochastic ergodic HJB equation,
motivated by the above formulation and Peng's finite-horizon
stochastic HJB equation \cite{Peng1992}:
\begin{equation}\label{eq:stochastic-ergodic-hjb-intro}
\lambda_t
=
\inf_{u\in\mathbb R^k}
\left\{
\ell_t(x,u)+\mathcal L_t^uV_t(x)
\right\}.
\end{equation}
The solution $(V,\lambda)$ consists of a semimartingale
random field $V$ and a real-valued process $\lambda$,
where $\mathcal L_t^u$ denotes the controlled
It\^o--Wentzell operator.

We use the infinite-horizon Riccati existence result
of \cite{GuatteriMasiero2010} and extend their existence result for the dual equation to bounded $f$ and $g$.
We establish uniqueness and uniform estimates for both equations.
These solutions yield a quadratic solution $(V,\lambda)$
of the stochastic ergodic HJB equation, for which we prove
uniqueness in a suitable sense.

In particular, we introduce the key uniformly local conditional
spaces $\mathbb L^p_{\mathrm{unif}}$, for $p\geq1$, defined by
$
\|Y\|_{\mathbb L^p_{\mathrm{unif}}}
:=\sup_{t\geq0}
\big\|\mathbb E\big[\int_t^{t+1}|Y_s|^p\,ds
\mid\mathcal F_t\big]\big\|_\infty^{1/p}.
$ We establish that $Q$ and $k$ in
\eqref{eq:riccati-dual-intro} belong to
$\mathbb L^2_{\mathrm{unif}}$ and that $\lambda$ in
\eqref{eq:stochastic-ergodic-hjb-intro} belongs to
$\mathbb L^1_{\mathrm{unif}}$,
providing the estimates needed for verification.

Following the ideas of
Lin, Lise, and Touzi~\cite{Lin2026},
we generalize  \eqref{eq:normalized-costs-intro}
 as
\[
J^{K_\rho}(x;\nu)
:=\mathbb E\int_0^\infty
K_\rho(t)\ell_t(X_t^{x;\nu},\nu_t)\,dt,
\qquad K_\rho(t):=\rho K(\rho t),
\]
where $K$ is nonnegative, nonincreasing, and c\`adl\`ag,
with $\int_0^\infty K(t)\,dt=1$. The choices $K(t)=\mathbf 1_{[0,1)}(t)$ and $K(t)=e^{-t}$
recover the two criteria in \eqref{eq:normalized-costs-intro},
respectively. For such $K$, we prove that
\[
\begin{aligned}
\liminf_{\rho\downarrow0}\inf_\nu J^{K_\rho}(x;\nu)
&=\inf_\nu\liminf_{\rho\downarrow0}J^{K_\rho}(x;\nu)
=\liminf_{\rho\downarrow0}\mathbb E\int_0^\infty K_\rho(t)\lambda_t\,dt,\\
\limsup_{\rho\downarrow0}\inf_\nu J^{K_\rho}(x;\nu)
&=\inf_\nu\limsup_{\rho\downarrow0}J^{K_\rho}(x;\nu)
=\limsup_{\rho\downarrow0}\mathbb E\int_0^\infty K_\rho(t)\lambda_t\,dt.
\end{aligned}
\]

We also establish the corresponding verification principle
and construct an optimal affine feedback control. Finally, for random-periodic coefficients, we construct
a random-periodic solution and show that the optimal
ergodic costs coincide for all these kernels.
We also establish stability of the Riccati and dual
solutions and the optimal ergodic costs under
perturbations of the coefficients.

The paper is organized as follows.
After introducing the framework in
Section~\ref{sec:framework-main-results} and presenting
our main results in Section~\ref{sec:main-results},
we prove the verification theorem in
Section~\ref{sec:proof_veri}.
Section~\ref{sec:finite-horizon} establishes well-posedness
of the infinite-horizon Riccati and dual equations and
of the stochastic ergodic HJB equation.
In Section~\ref{sec:further-characterizations}, we study
random-periodic coefficients and establish stability
with respect to the coefficients.
\section{Framework and problem formulation}
\label{sec:framework-main-results}
 \subsection{Notation}

Let \(W=(W^1,\ldots,W^d)\) be a \(d\)-dimensional Brownian motion on a
complete probability space \((\Omega,\mathcal F,\mathbb P)\), and let
\(\mathbb F=(\mathcal F_t)_{t\geq0}\) be its augmented natural filtration.
We denote by \(\mathbb E_t:=\mathbb E(\,\cdot\mid\mathcal F_t)\) the
conditional expectation with respect to \(\mathcal F_t\).
Throughout the paper,
unless otherwise stated, all processes are assumed to be progressively
measurable with respect to $\mathbb F$. Let
$\mathcal S^n:=\{A\in\mathbb R^{n\times n}:A=A^\top\}$ and
$\mathcal S_+^n:=\{A\in\mathcal S^n:x^\top Ax\geq0
\text{ for all }x\in\mathbb R^n\}$. Let $E$ be a finite-dimensional Euclidean space, including
the vector, matrix, and product spaces used below, endowed with its usual
Euclidean  norm, denoted by $|\cdot|$. We introduce the
following spaces.
\begin{itemize}

\item $\mathbb L^\infty(E)$ is the space of progressively measurable processes
$Y:\Omega\times\mathbb R_+\to E$ such that
$\|Y\|_{\mathbb L^\infty(E)}:=\operatorname*{ess\,sup}_{(\omega,t)}
|Y_t(\omega)|<\infty$,
where the essential supremum is taken with respect to
$d\mathbb P\otimes dt$.

\item $\mathbb S^\infty(E)$ is the subspace of
$\mathbb L^\infty(E)$ consisting of processes with
$\mathbb P$-a.s.\ continuous paths. For $T>0$,
$\mathbb S_T^\infty(E)$ is defined similarly, with $\mathbb R_+$
replaced by $[0,T]$.

\item For $T>0$, $\mathbb H_T^2(E)$ is the space of progressively measurable processes
$Y:\Omega\times[0,T]\to E$ such that
$\|Y\|_{\mathbb H_T^2(E)}:=\bigl(\mathbb E\int_0^T|Y_t|^2\,dt\bigr)^{1/2}<\infty$.
We set
$\mathbb H_\infty^2(E):=\bigcap_{T>0}\mathbb H_T^2(E)$.

\item For $p\geq1$, $\mathbb L^p_{\mathrm{unif}}(E)$ is the space of
progressively measurable processes $Y:\Omega\times\mathbb R_+\to E$ such that
$\|Y\|_{\mathbb L^p_{\mathrm{unif}}(E)}:=\sup_{t\in\R_+}\left\|\mathbb E_t\int_t^{t+1}|Y_s|^p\,ds\right\|_\infty^{1/p}<\infty$.

\item Let $\mathcal C_{\mathrm{sm}}^2$ be the space of semimartingale
random fields $V=(V_t(\cdot))_{t\geq0}$ s.t.
\begin{equation}\label{eq:semimartingale-random-field}
dV_t(x)=a_t(x)\,dt+\sum_{i=1}^d Z_t^i(x)\,dW_t^i,
\qquad \text{for every }x\in\mathbb R^n,
\end{equation}
where $V_t(\cdot)\in C^2(\mathbb R^n)$,
$Z_t^i(\cdot)\in C^1(\mathbb R^n)$,
$a_t(\cdot)\in C(\mathbb R^n)$, and, for every
$\Phi\in\{V,D_xV,D_x^2V\}$ and
$\Psi\in\{a,|Z^i|^2,|D_xZ^i|^2\}$, $i=1,\ldots,d$,
\[
\mathbb E\sup_{\substack{0\leq t\leq T\\|x|\leq R}}|\Phi_t(x)|
+\mathbb E\int_0^T\sup_{|x|\leq R}|\Psi_t(x)|\,dt
<\infty,
\qquad \text{for every }T,R>0.
\]
All the random fields involved are  progressively measurable
for all $x\in\mathbb R^n$.

\end{itemize}

For later use, we recall the It\^o--Wentzell operator. Let $X$ be an
$\mathbb R^n$-valued It\^o process of the form
$
dX_t=b_t\,dt+\sum_{i=1}^d\sigma_t^i\,dW_t^i.
$
For $V\in\mathcal C_{\mathrm{sm}}^2$ with decomposition
\eqref{eq:semimartingale-random-field}, the It\^o--Wentzell formula gives
\begin{gather}
dV_t(X_t)
=
\mathcal L_tV_t(X_t)\,dt
+\sum_{i=1}^d
\left[
Z_t^i(X_t)+D_xV_t(X_t)\cdot\sigma_t^i
\right]dW_t^i,
\quad\text{where}
\notag\\
\mathcal L_tV_t(x):=
a_t(x)+D_xV_t(x)\cdot b_t
+\sum_{i=1}^d
\left[
\frac12\bigl(D_x^2V_t(x)\sigma_t^i\bigr)\cdot\sigma_t^i
+D_xZ_t^i(x)\cdot\sigma_t^i
\right].
\label{eq:ito-wentzell-operator}
\end{gather}

\subsection{Forward dynamics and ergodic control problem}
\label{subsec:forward-dynamics-control}
The admissible control set is
$\mathcal U:=\mathbb H_\infty^2(\mathbb R^k)$. For $t\geq0$, $x\in\mathbb R^n$, and $\nu\in\mathcal U$, we consider
the controlled linear stochastic system for $s\geq t$ and running cost
\begin{align}
dX_s^{t,x;\nu}
&=
\left(A_sX_s^{t,x;\nu}+B_s\nu_s+f_s\right)\,ds
+\sum_{i=1}^d
\left(C_s^iX_s^{t,x;\nu}+D_s^i\nu_s+g_s^i\right)\,dW_s^i
 ,
\label{eq:controlled-LQ-dynamics}\\
X_t^{t,x;\nu}
&=x,
\qquad \text{and}\qquad
\ell_t(x,u):=S_tx\cdot x+|u|^2,
\quad (x,u)\in\mathbb R^n\times\mathbb R^k.
\notag
\end{align}
where the coefficients $A,B,C,D,f,g$, together with the quadratic
state cost coefficient $S$ of the running cost, satisfy the following
assumptions.

\begin{assumption}
\label{ass:bounded-coefficients}
For each $i=1,\ldots,d$,
$A,C^i\in\mathbb L^\infty(\mathbb R^{n\times n})$,
$B,D^i\in\mathbb L^\infty(\mathbb R^{n\times k})$,
$f,g^i\in\mathbb L^\infty(\mathbb R^n)$, and
$S\in\mathbb L^\infty(\mathcal S_+^n)$.
\end{assumption}

For a problem starting at time $t$, only the restrictions of $X$ and
$\nu$ to $[t,\infty)$ matter. Throughout the paper, we therefore leave
their values before $t$ unspecified.

Under Assumption~\ref{ass:bounded-coefficients}, for every
$\nu\in\mathcal U$, $t\geq0$, and $x\in\mathbb R^n$, the state equation
\eqref{eq:controlled-LQ-dynamics} admits a unique continuous strong solution
$X^{t,x;\nu}$ satisfying
$\mathbb E\big[\sup_{t\leq s\leq T}|X_s^{t,x;\nu}|^2\big]<\infty$
for every $T\geq t$. With the boundedness of the coefficients, these estimates make the
localization arguments used below standard, and we omit their routine details.

For the controlled dynamics \eqref{eq:controlled-LQ-dynamics} and
$u\in\mathbb R^k$, we define $\mathcal L_t^u$ by replacing $b_t$ and
$\sigma_t^i$ in \eqref{eq:ito-wentzell-operator} with
$A_tx+B_tu+f_t$ and $C_t^ix+D_t^iu+g_t^i$, respectively.
Thus, for every $\nu\in\mathcal U$ and
$V\in\mathcal C_{\mathrm{sm}}^2$, the drift coefficient of
$V_s(X_s^{t,x;\nu})$ is
$\mathcal L_s^{\nu_s}V_s(X_s^{t,x;\nu})$.

For $0\le t\le T\le\infty$ with $t<\infty$, the associated cost is defined by
\begin{equation}\label{eq:horizon-cost}
    J_t^T(x;\nu)
    :=
    \mathbb E_t
    \int_t^T
    \ell_s(X_s^{t,x;\nu},\nu_s)\,ds
    =
    \mathbb E_t
    \int_t^T
    \left(
        S_sX_s^{t,x;\nu}\cdot X_s^{t,x;\nu}
        +
        |\nu_s|^2
    \right)ds.
\end{equation}

When $t=0$, we simply write
$X^{x;\nu}:=X^{0,x;\nu}$ and $J^T(x;\nu):=J_0^T(x;\nu)$.
For every $\nu\in\mathcal U$, the state process satisfies
$X^{t,x;\nu}\in\mathbb H_\infty^2(\mathbb R^n)$.
Hence, by the boundedness of $S$,
$J_t^T(x;\nu)<\infty$, $\mathbb P$-a.s., for every
$0\leq t\leq T<\infty$.

Following \cite{Lin2026}, we introduce a general class of averaging
kernels:
\[
\mathcal K
:=
\left\{
K:\mathbb R_+\to\mathbb R_+:
K\text{ is nonincreasing and c\`adl\`ag},
\quad
\int_0^\infty K(t)\,dt=1
\right\}.
\]
For $K\in\mathcal K$ and $\rho>0$, define
\[
K_\rho(t):=\rho K(\rho t),
\qquad
J^{K_\rho}(x;\nu)
:=
\mathbb E\int_0^\infty
K_\rho(t)\ell_t(X_t^{x;\nu},\nu_t)\,dt
\in[0,\infty].
\]

The Ces\`aro kernel $K^0(t):=\mathbf 1_{[0,1)}(t)$
and the exponential kernel $K^{\rm exp}(t):=e^{-t}$
recover the normalized finite-horizon and discounted costs
in \eqref{eq:normalized-costs-intro}, respectively.

The lower and upper ergodic costs associated with $K$ are defined by
\[
J_*(K;x,\nu)
:=
\liminf_{\rho\downarrow0}J^{K_\rho}(x;\nu),
\qquad
J^*(K;x,\nu)
:=
\limsup_{\rho\downarrow0}J^{K_\rho}(x;\nu).
\]

Our objective is to characterize the infima of these costs over $\nu\in\mathcal U$,
to establish a verification principle, and to construct admissible
controls attaining these infima.

\begin{remark}
In several treatments of ergodic LQ control, the optimization is
restricted to a subclass of affine feedback controls satisfying
additional stability conditions on the corresponding closed-loop
systems; see, e.g., Mei, Wei, and Yong~\cite{MeiWeiYong2021}
and \cite{WuZhang2026}.
These conditions ensure the transversality conditions used in the
verification argument.
Here, we consider the general open-loop admissible class $\mathcal U$,
without imposing such conditions in advance.
Under the uniform coercivity assumption on $S$ stated below,
the verification argument applies to this entire $\mathcal U$; see
Theorem~\ref{thm:general-decreasing-kernels}.
\end{remark}

\subsection{Uniform stabilizability}
In classical ergodic control problems, dissipativity is a standard
condition used to obtain stability estimates for the state dynamics;
see, e.g., Fuhrman, Hu, and Tessitore~\cite{FuhrmanHuTessitore2009}.
In the linear--quadratic setting, following \cite{GuatteriMasiero2010}, we impose uniform stabilizability, formulated through a uniform
quadratic bound on the optimal infinite-horizon cost of the homogeneous
problem. We therefore consider the homogeneous problem obtained by setting
the affine terms $f$ and $g$ equal to zero. For $t\geq0$, $x\in\mathbb R^n$, and $\nu\in\mathcal U$, let
$\overline{X}^{t,x;\nu}$ denote the solution, for $s\geq t$, of
\begin{equation}\label{eq:homogeneous-controlled-dynamics}
d\overline{X}_s^{t,x;\nu}
=
\left(A_s\overline{X}_s^{t,x;\nu}+B_s\nu_s\right)\,ds
+
\sum_{i=1}^d
\left(C_s^i\overline{X}_s^{t,x;\nu}+D_s^i\nu_s\right)\,dW_s^i,
\quad \overline{X}_t^{t,x;\nu}=x.
\end{equation}
We define its infinite-horizon cost by
\begin{equation}\label{eq:homogeneous-infinite-horizon-cost}
\overline{J}_t^\infty(x;\nu)
:=
\mathbb E_t\int_t^\infty
\left(
S_s\overline{X}_s^{t,x;\nu}\cdot\overline{X}_s^{t,x;\nu}
+
|\nu_s|^2
\right)ds.
\end{equation}

\begin{assumption}
\label{ass:uniform-stabilizability}
There exist $M,\delta>0$ such that
$S_t\geq\delta I_n$, $d\mathbb P\otimes dt$-a.e., and
\begin{equation}\label{eq:uniform_stability}
\operatorname*{ess\,inf}_{\nu\in\mathcal U}
\overline J_t^\infty(x;\nu)
\leq M|x|^2,
\quad\mathbb P\text{-a.s., for every }t\geq0,\ x\in\mathbb R^n.
\end{equation}
\end{assumption}

This is a concise
formulation of the conditions required in \cite{GuatteriMasiero2010}.
Note that \cite[Definition~3.3 and Remark~3.5]{GuatteriMasiero2010}
allow $M$ in \eqref{eq:uniform_stability} to depend on $t$ to ensure the existence of a minimal nonnegative solution
$P$ to the infinite-horizon backward stochastic Riccati equation, see \cite[Corollary~3.7]{GuatteriMasiero2010}.
However, they require the boundedness of $P$ in \cite[Theorem~4.3]{GuatteriMasiero2010} and the stability condition in
\cite[Definition~4.1]{GuatteriMasiero2010} to establish the existence of solution for the dual equation. Under the uniform coercivity of $S$, these conditions are equivalent to \eqref{eq:uniform_stability}.

\begin{remark}\label{rem:sufficient-stability}
Arguing as in \cite[Remark~4.2]{GuatteriMasiero2010},
\eqref{eq:uniform_stability} holds under the dissipativity condition
$A_t^\top+A_t+\sum_{i=1}^d(C_t^i)^\top C_t^i
\leq-\beta I_n$, $d\mathbb P\otimes dt$-a.e.,
for some $\beta>0$.
\end{remark}

\subsection{Stochastic ergodic HJB formulation}
\label{subsec:ergodic-equation}

To study ergodic control problems with stochastic coefficients, we introduce
the stochastic ergodic HJB equation, motivated by the ergodic
HJB equation \cite{ArisawaLions1998,BensoussanFrehse1992} and stochastic
HJB equation of \cite{Peng1992}:
\begin{equation}\label{eq:stochastic-ergodic-hjb}
\lambda_t
=
\inf_{u\in\mathbb R^k}
\left\{
\ell_t(x,u)+\mathcal L_t^uV_t(x)
\right\},
\end{equation}
where $V$ is a semimartingale random field, and $\lambda$ is
a real-valued process, replacing the constant in the classical
ergodic HJB equation.

\begin{definition}
\label{def:ergodic_solution}
A pair $(V,\lambda)$ is called a solution of the stochastic ergodic HJB
equation if
$V\in\mathcal C_{\mathrm{sm}}^2$,
$\lambda\in\mathbb L^1_{\mathrm{unif}}(\mathbb R)$,
and, $d\mathbb P\otimes dt$-a.e.,
\eqref{eq:stochastic-ergodic-hjb} holds for all $x\in\mathbb R^n$,
and  
\begin{equation}\label{eq:growth}
-C\leq V_t(x)\leq C(1+|x|^2),~ t\geq0,~x\in\mathbb R^n,
~\mathbb P\text{-a.s. for some constant}
~C>0.
\end{equation}
\end{definition}

The quadratic growth in~\eqref{eq:growth} is natural for an LQ control problem, whose value field
is actually expected to be quadratic--affine in the state variable.
However, Remark~\ref{rem:verification-lower-bound} below shows that
even two-sided quadratic growth bounds do not ensure that
$\lambda$ characterizes the optimal ergodic costs,
and illustrates the role of the uniform lower bound for $V$
in Condition~\eqref{eq:growth}.

\begin{remark}
The formulation above is complementary to the path-dependent ergodic control
framework of \cite{Lin2026}. There, the problem is studied in a weak formulation with
time- and path-dependent coefficients without exogenous randomness,
leading naturally to a BSDE formulation with solution
$(Y,Z,\Lambda)$, where $\Lambda$ is deterministic. Here,
we use a strong formulation with random coefficients and state-dependent
dynamics, leading instead to a solution $(V,\lambda)$ of the stochastic
ergodic HJB equation.

Formally, $V$ corresponds to $Y$, while the martingale term in the
It\^o--Wentzell formula corresponds to $Z$. Moreover,
$\E\int_t^T\lambda_s\,ds$ corresponds to $\Lambda(T)-\Lambda(t)$. In
\cite{Lin2026},  $\Lambda$ may be chosen absolutely
continuous with a bounded density. In the present LQ setting, however, 
$\E[\lambda_t]$ associated with the canonical
quadratic solution need not be bounded. This
motivates the more general requirement
$\lambda\in\mathbb L^1_{\mathrm{unif}}(\mathbb R)$, which is
weak enough to include the canonical solution while still providing the
integrability needed for the verification argument and for averaging with
respect to the general kernels.
\end{remark}

\section{Main results}\label{sec:main-results}
\subsection{Well-posedness of the stochastic ergodic HJB equation}

Throughout this paper, solutions are understood in the sense of
Definition~\ref{def:ergodic_solution}.

\begin{theorem}
\label{thm:existence_ergodic_hjb}
Let Assumptions~\ref{ass:bounded-coefficients} and
\ref{ass:uniform-stabilizability} hold. Then

\textup{(i)}
The stochastic ergodic HJB equation admits a quadratic solution
$(V^0,\lambda^0)$ satisfying, for some $\delta_P>0$:
\begin{equation}\label{eq:quadratic-affine-ergodic-solution}
V_t^0(x)=P_tx\cdot x+2r_t\cdot x,
~~\text{with }~
r\in\mathbb S^\infty(\mathbb R^n),
~
P\in\mathbb S^\infty(\mathcal S_+^n),
~
P\geq\delta_P I_n,
~\mathbb P-\text{a.s.}
\end{equation}

\textup{(ii)}
For every solution $(V,\lambda)$, we have
$
\sup_{0\leq t\leq T<\infty}
\Big\|
\mathbb E_t\int_t^T
(\lambda_s^0-\lambda_s)\,ds
\Big\|_\infty
<\infty.
$

\textup{(iii)} If, in addition, $V$ is uniformly coercive
and locally uniformly continuous in $x$, i.e., for some
$\widetilde c>0$ and for all $R>0$, $\mathbb P$-a.s. 
\begin{equation}\label{eq:uniform-local-continuity}
\operatorname*{essinf}_{\omega\in\Omega}
\!\inf_{\substack{t\geq0,x\in\mathbb R^n}}
\bigl(V_t(\omega,x)-\widetilde c|x|^2\bigr)
>-\infty,
\qquad
\lim_{\eta\downarrow0}\varpi_V^R(\eta)=0,
\end{equation}
where $\varpi_V^R$ denotes the spatial modulus of continuity of $V$
on the ball of radius $R$ centered at the origin,
uniformly over $t\geq0$,
then $V_t^0(x)-V_t(x)
-\int_0^t(\lambda_s^0-\lambda_s)\,ds$
does not depend on $x$ and is a local martingale.
\end{theorem}

The proof of \textup{(i)} is reported in
Section~\ref{subsec:quadratic-solution}, and the proofs of
\textup{(ii)} and \textup{(iii)} are reported in
Section~\ref{sec:proof_ergodic_HJB}.

The coercivity condition in
Theorem~\ref{thm:existence_ergodic_hjb}\textup{(iii)} is the natural LQ
version of the condition used in \cite{BensoussanFrehse1992}, where the
solution is required to tend to $+\infty$ as $|x|\to\infty$.
The following example shows that the uniform local
continuity assumption is also essential.

\begin{example}
Let $dX_t=-X_t\,dt$ and $\ell(x,u)=x^2+u^2$.
Then, for any $\phi\in C_b^2(\mathbb R)$,
$(V^\phi_t(x),\lambda_t)
=\left(\frac12x^2+\phi(e^tx),0\right)$
is a coercive solution of the stochastic ergodic HJB equation
in the sense of Definition~\ref{def:ergodic_solution}, since
$\partial_t\phi(e^tx)-xD_x\phi(e^tx)=0$.
If $\phi$ is nonconstant, then $V^\phi$ is not locally uniformly
continuous in $x$, uniformly in $t$, and
$V^\phi_t(x)-V^0_t(x)=\phi(e^tx)$ is not constant in $x$.
\end{example}

\begin{remark}
\label{rem:additional-uniqueness-assumptions}
A convenient sufficient condition for the uniform local continuity in
Theorem~\ref{thm:existence_ergodic_hjb}~\textup{(iii)} is the convexity of $V_t$.
Indeed, if $x\mapsto V_t(x)$ is convex, then the growth
bounds in \eqref{eq:growth} imply
$
|D_xV_t(x)|\leq C'(1+|x|)
$
for some constant $C'>0$.
Thus $V$ is locally Lipschitz in $x$, uniformly in $t$,
and satisfies the continuity condition in
Theorem~\ref{thm:existence_ergodic_hjb}~\textup{(iii)}.
\end{remark}

\subsection{Structure of the quadratic solution}
\label{subsec:quadratic-solution}
The coefficients $P$ and $r$ in the quadratic solution
\eqref{eq:quadratic-affine-ergodic-solution} are constructed
from the infinite-horizon backward stochastic Riccati equation
and the associated dual BSDE, building on
\cite{GuatteriMasiero2010}.
We use their existence result for the Riccati equation and
adapt their arguments to establish existence for the dual
equation in our setting, obtaining further properties
of both solutions.

For $p\in\mathcal S_+^n$, $q=(q^1,\ldots,q^d)\in(\mathcal S^n)^d$,
$r\in\mathbb R^n$, and $k=(k^1,\ldots,k^d)\in(\mathbb R^n)^d$,
define the following maps, with $i=1,\ldots,d$:
\begin{gather}
N_t(p):=I_k+\sum_{i=1}^d(D_t^i)^\top pD_t^i,
\qquad
L_t(p,q):=B_t^\top p+\sum_{i=1}^d(D_t^i)^\top(pC_t^i+q^i),
\label{eq:LQ-coefficient-maps}\\
\Gamma_t(p,q)\!:=\!-N_t(p)^{-1}L_t(p,q),
H_t(p,q)\!:=\!A_t+B_t\Gamma_t(p,q),
M_t^i(p,q)\!:=\!C_t^i+D_t^i\Gamma_t(p,q),
\notag\\
\varphi_t(p,q):=pf_t+\sum_{i=1}^d
\bigl(M_t^i(p,q)^\top p+q^i\bigr)g_t^i,
\quad
\eta_t(p,r,k):=B_t^\top r+\sum_{i=1}^d(D_t^i)^\top(pg_t^i+k^i),
\notag\\
\mathcal R_t(p,q):=
A_t^\top p+pA_t
+\sum_{i=1}^d\bigl((C_t^i)^\top pC_t^i
+(C_t^i)^\top q^i+q^iC_t^i\bigr)
+S_t+L_t(p,q)^\top \Gamma_t(p,q),
\notag\\
\mathcal G_t(p,q,r,k):=
H_t(p,q)^\top r+\sum_{i=1}^dM_t^i(p,q)^\top k^i
+\varphi_t(p,q).\notag
\end{gather}

We consider the infinite-horizon backward stochastic Riccati equation
and its associated dual BSDE:
\begin{align}
dP_t
&=-\mathcal R_t(P_t,Q_t)\,dt
+\sum_{i=1}^dQ_t^i\,dW_t^i,
\qquad t\geq0,
\label{eq:infinite-horizon-riccati}\\
dr_t
&=-\mathcal G_t(P_t,Q_t,r_t,k_t)\,dt
+\sum_{i=1}^dk_t^i\,dW_t^i,
\qquad t\geq0.
\label{eq:infinite_dual}
\end{align}

\begin{proposition}
\label{prop:infinite-horizon-riccati}
Let Assumptions~\ref{ass:bounded-coefficients} and
\ref{ass:uniform-stabilizability} hold.

\textup{(i)}
Equation~\eqref{eq:infinite-horizon-riccati} admits a unique solution
$(P,Q)\in\mathbb S^\infty(\mathcal S_+^n)
\times\mathbb H_\infty^2((\mathcal S^n)^d)$.
Moreover, $Q\in\mathbb L^2_{\mathrm{unif}}((\mathcal S^n)^d)$ and
$P\geq\delta_P I_n$, $\mathbb P$-a.s., for some $\delta_P>0$.

\textup{(ii)}
For $(P,Q)$ given by \textup{(i)},
equation~\eqref{eq:infinite_dual} admits a unique solution
$(r,k)\in\mathbb S^\infty(\mathbb R^n)
\times\mathbb H_\infty^2((\mathbb R^n)^d)$.
Moreover, $k\in\mathbb L^2_{\mathrm{unif}}((\mathbb R^n)^d)$.
\end{proposition}

The proof is reported in Section~\ref{sec:existence_ergodic_HJB}. In what follows, $(P,Q,r,k)$ denotes the solution given by
Proposition~\ref{prop:infinite-horizon-riccati}.
For simplicity, we write
$N_t:=N_t(P_t)$, $L_t:=L_t(P_t,Q_t)$,
and $\eta_t:=\eta_t(P_t,r_t,k_t)$. 
For $t\geq0$ and $x\in\mathbb R^n$, we define
\begin{equation}\label{eq:quadratic-ergodic-value-field}
V_t^0(x):=P_tx\cdot x+2r_t\cdot x,
~~
\lambda_t^0:=2f_t\cdot r_t
+\sum_{i=1}^d\left(P_tg_t^i\cdot g_t^i+2k_t^i\cdot g_t^i\right)
-N_t^{-1}\eta_t\cdot\eta_t.
\end{equation}

\begin{proof}[Proof of Theorem~\ref{thm:existence_ergodic_hjb} \textup{(i)}]
Let $V^0$ and $\lambda^0$ be defined by
\eqref{eq:quadratic-ergodic-value-field}.
The martingale coefficients of $V^0$ are
$Z_t^{0,i}(x)=Q_t^ix\cdot x+2k_t^i\cdot x$, $i=1,\ldots,d$.
Using \eqref{eq:infinite-horizon-riccati} and
\eqref{eq:infinite_dual} and completing the square,
as in the proof of \cite[Theorem~5.2]{GuatteriMasiero2010},
we obtain, $d\mathbb P\otimes dt$-a.e.,
\begin{equation}\label{eq:completion-of-squares}
\ell_t(x,u)+\mathcal L_t^uV_t^0(x)-\lambda_t^0
=
\left|N_t^{\frac{1}{2}}\bigl(u+N_t^{-1}(L_tx+\eta_t)\bigr)\right|^2,
\,\text{for }(x,u)\in\mathbb R^n\times\mathbb R^k,
\end{equation}
Taking the infimum over $u$ yields
\eqref{eq:stochastic-ergodic-hjb}.
The equations and estimates in
Proposition~\ref{prop:infinite-horizon-riccati} \textup{(i)--(ii)} imply
$V^0\in\mathcal C_{\mathrm{sm}}^2$ and
$\lambda^0\in\mathbb L^1_{\mathrm{unif}}(\mathbb R)$.
Moreover, the uniform coercivity of $P$ and the boundedness of
$P$ and $r$ give $-C\leq V_t^0(x)\leq C(1+|x|^2)$.
Thus, $(V^0,\lambda^0)$ is a solution of the stochastic ergodic HJB equation in the sense of
Definition~\ref{def:ergodic_solution}.
\end{proof}

\subsection{Ergodic values and verification}
\label{subsec:ergodic-values-verification}

In this subsection, we characterize the optimal ergodic costs through
the process $\lambda$ and establish a verification principle for optimal
controls. For $K\in\mathcal K$ and
$\lambda\in\mathbb L^1_{\mathrm{unif}}(\mathbb R)$, define
\begin{equation*}
\Lambda^K
:=
\mathbb E\int_0^\infty K(t)\lambda_t\,dt
=
-\int_{(0,\infty)}\Lambda(t)\,dK(t),
\end{equation*}
where $\Lambda(t):=\mathbb E\int_0^t\lambda_s\,ds$.
The last equality follows from Lebesgue--Stieltjes integration by parts,
noting that $\Lambda$ is absolutely continuous with
$\Lambda'(t)=\mathbb E[\lambda_t]$ for almost every $t\geq0$,
has at most linear growth, and satisfies $\Lambda(0)=0$,
while $\lim_{t\to\infty}tK(t)=0$ by the monotonicity and
integrability of $K$.

The lower and upper ergodic values associated with $\lambda$ and $K\in\K$
are defined by:
\[
\lambda_*(K):=\liminf_{\rho\downarrow0}\Lambda^{K_\rho},
\qquad
\lambda^*(K):=\limsup_{\rho\downarrow0}\Lambda^{K_\rho}.
\]

These quantities are independent of the choice of the stochastic ergodic
HJB solution. Let $(V^0,\lambda^0)$ be the quadratic
solution, and let $\Lambda^0,\Lambda^{0,K}$ denote
the corresponding quantities.
Taking $t=0$ in Theorem~\ref{thm:existence_ergodic_hjb}\textup{(ii)}
gives $\|\Lambda-\Lambda^0\|_\infty<\infty$. Then
\[
\left|\Lambda^{K_\rho}-\Lambda^{0,K_\rho}\right|
\leq K_\rho(0)\|\Lambda-\Lambda^0\|_\infty
=\rho K(0)\|\Lambda-\Lambda^0\|_\infty
\xrightarrow[\rho\downarrow0]{}0.
\]

Hence $\lambda_*(K)$ and $\lambda^*(K)$ are independent of the choice of
$(V,\lambda)$. To compare the ergodic values associated with different kernels,
define
\[
\mathcal K_W
:=\left\{
K\in\mathcal K\cap\mathrm{AC}:
\mathscr F\bigl[e^{-2u}K'(e^{-u})\bigr](q)\neq0
\text{ for all }q\in\mathbb R
\right\},
\]
where $\mathrm{AC}$ denotes the class of absolutely continuous
functions on $\mathbb R_+$ and $\mathscr F$ denotes the Fourier
transform. In particular, $K^{\rm exp}\in\mathcal K_W$, since
$
\mathscr F\bigl[e^{-2u}(K^{\rm exp})'(e^{-u})\bigr](q)
=-\Gamma(2+iq)\neq0.
$

Recall that $K^0:=\mathbf 1_{[0,1)}$.
The following proposition gives the analogue of
Proposition~2.2 in \cite{Lin2026} for a stochastic process
$\lambda\in\mathbb L^1_{\mathrm{unif}}(\mathbb R)$.

\begin{proposition}
\label{prop:relation-value}
For all $K\in\mathcal K$, we have
$\lambda_*(K^0)\leq\lambda_*(K)\leq\lambda^*(K)\leq\lambda^*(K^0)$.
Moreover, if $\lambda_*(\widehat K)=\lambda^*(\widehat K)$
for some $\widehat K\in\{K^0\}\cup\mathcal K_W$, then
$\lambda_*=\lambda^*$ on $\mathcal K$.
\end{proposition}

The proof is reported in Appendix~\ref{app:averaging-kernels}.

\begin{remark}
Every bounded nonnegative progressively measurable process $\lambda$
can be realized in our setting: take $n=k=d=1$,
$A=-1$, $B=1$, $S=3$, $C=D=f=0$, and $g_t=\sqrt{\lambda_t}$.
These coefficients satisfy our assumptions, and
$V_t(x)=x^2$ gives the desired solution $(V,\lambda)$.
In particular, this construction realizes the configurations
discussed in \cite[Remark~2.4]{Lin2026} for the Ces\`aro and
exponential kernels. 
\end{remark}

\begin{definition}
\label{def:verification-controls}
Let $(V,\lambda)$ be a solution of the stochastic ergodic HJB equation
in the sense of Definition~\ref{def:ergodic_solution}. For $x\in\R^n$, define $F_t^uV_t(x):=\ell_t(x,u)+\mathcal L_t^uV_t(x)$ and
\[
\mathcal U_{\mathrm{ver}}^{V,\lambda}(x)
:=
\Big\{
\nu\in\mathcal U:
\lim_{T\to\infty}\frac1T\mathbb E\int_0^T
\left[F_t^{\nu_t}V_t(X_t^{x;\nu})-\lambda_t\right]dt=0
\Big\}.
\]
\end{definition}

\begin{theorem}
\label{thm:general-decreasing-kernels}
Let Assumptions~\ref{ass:bounded-coefficients} and
\ref{ass:uniform-stabilizability} hold.
Then, for $K\in\mathcal K$ and $x\in\mathbb R^n$,
\begin{align}
\liminf_{\rho\downarrow0}\inf_{\nu\in\mathcal U}J^{K_\rho}(x;\nu)
&=
\inf_{\nu\in\mathcal U}\liminf_{\rho\downarrow0}J^{K_\rho}(x;\nu)
=
\lambda_*(K),
\label{eq:kernel-lower-value}\\
\limsup_{\rho\downarrow0}\inf_{\nu\in\mathcal U}J^{K_\rho}(x;\nu)
&=
\inf_{\nu\in\mathcal U}\limsup_{\rho\downarrow0}J^{K_\rho}(x;\nu)
=
\lambda^*(K).
\label{eq:kernel-upper-value}
\end{align}
Moreover: \textup{(i)} For any solution $(V,\lambda)$ in the sense of
Definition~\ref{def:ergodic_solution},
$\mathcal U_{\mathrm{ver}}^{V,\lambda}(x)\neq\emptyset$,
and every $\nu\in\mathcal U_{\mathrm{ver}}^{V,\lambda}(x)$
attains the infima of the lower and upper ergodic costs in
\eqref{eq:kernel-lower-value} and \eqref{eq:kernel-upper-value}
for the initial state $x$.

\textup{(ii)} Define the affine feedback map
$
u_t^*(x):=-N_t^{-1}(L_tx+\eta_t).
$
For every $x\in\mathbb R^n$, the corresponding closed-loop
equation with initial condition $X_0^{x,*}=x$ admits a unique
global strong solution $X^{x,*}$, and the induced control
$\nu_t^{x,*}:=u_t^*(X_t^{x,*})$, $t\geq0$, belongs to
$\mathcal U_{\mathrm{ver}}^{V^0,\lambda^0}(x)$.
Consequently, for every $K\in\mathcal K$, this control attains
the infima of the lower and upper ergodic costs in
\eqref{eq:kernel-lower-value} and \eqref{eq:kernel-upper-value}.
\end{theorem}

The proof is reported in Section~\ref{sec:proof_veri}.

\begin{remark}\label{rem:verification-lower-bound}
The uniform lower bound in \eqref{eq:growth} is essential for
characterizing the optimal ergodic costs through $\lambda$.
Indeed, consider the example $dX_t=\nu_t\,dt+dW_t$ and
$\ell(x,u)=x^2+u^2$.
Both $(V^0,\lambda^0)=(x^2,1)$ and
$(V,\lambda)=(-x^2,-1)$
satisfy all the requirements of Definition~\ref{def:ergodic_solution}
other than the uniform lower bound.
The two distinct constants $1$ and $-1$ cannot both represent
the optimal ergodic cost.
In fact, $(V^0,\lambda^0)=(x^2,1)$ satisfies \eqref{eq:growth},
and the optimal ergodic cost is $1$, whereas $V(x)=-x^2$ is not uniformly bounded below.
\end{remark}

\section{Proof of the verification theorem}\label{sec:proof_veri}

In this section, we prove Theorem~\ref{thm:general-decreasing-kernels}.
In particular, in view of our problem formulation using strong solutions of the controlled state SDE, we construct locally Lipschitz approximately optimal feedback maps and we show strong wellposedness of the
corresponding closed-loop equations.

\begin{lemma}
\label{lem:existence-verification-controls}
Let Assumptions~\ref{ass:bounded-coefficients} and
\ref{ass:uniform-stabilizability} hold, and let $(V,\lambda)$ be a
solution of the stochastic ergodic HJB equation.
For every $t\geq0$ and $\varepsilon>0$, there exists a
progressively measurable feedback map $u^{\varepsilon,t}$ such that,
for every $x\in\mathbb R^n$, the corresponding closed-loop equation
admits a unique global strong solution $X^{\varepsilon,t,x}$ satisfying:
\begin{equation}\label{eq:verification-control-bound}
\mathbb E_t\int_t^\infty
\left[F_s^{\nu_s^{\varepsilon,t,x}}V_s(X_s^{\varepsilon,t,x})-\lambda_s\right]ds
\leq\varepsilon,
\,\mathbb P\text{-a.s. and}~
\nu^{\varepsilon,t,x}:=u^{\varepsilon,t}(X^{\varepsilon,t,x})
\in\mathcal U.
\end{equation}
\end{lemma}

\begin{proof}
Fix $t\geq0$ and $\varepsilon>0$, and set
$G_s(x,u):=F_s^uV_s(x)-\lambda_s$.
By \eqref{eq:stochastic-ergodic-hjb},
$\inf_{u\in\mathbb R^k}G_s(x,u)=0$.
Since $V\in\mathcal C_{\mathrm{sm}}^2$,
$G$ admits the representation
$
G_s(x,u)=u^\top M_s(x)u+2b_s(x)\cdot u+c_s(x),
$
where $M,b,c$ can be expressed in terms of the model
coefficients, $\lambda$, and the semimartingale decomposition
and spatial derivatives of $V$.
Together with the integrability assumptions on $(V,\lambda)$,
this allows us to apply
Lemma~\ref{lem:blockwise-integrated-lipschitz-selector},
which yields a progressively measurable feedback map
$u^{\varepsilon,t}$, locally bounded and Lipschitz in $x$,
such that, for every progressively measurable
$\mathbb R^n$-valued process $Y$,
\begin{equation}\label{eq:approximate-feedback-selection}
0\leq\mathbb E_t\int_t^\infty
\left[F_s^{u_s^{\varepsilon,t}(Y_s)}V_s(Y_s)-\lambda_s\right]ds
\leq\varepsilon. 
\end{equation}

Fix $x\in\mathbb R^n$, and let $X^{\varepsilon,t,x}$ be the corresponding
maximal closed-loop solution with explosion time $\zeta$.
Define $\nu^{\varepsilon,t,x}$ as in the statement and set
$\tau_N:=\inf\{s\geq t:|X_s^{\varepsilon,t,x}|\geq N\}$.
By \eqref{eq:approximate-feedback-selection} and the nonnegativity of $G$,
for every $T>t$,
\begin{equation}\label{eq:stopped-feedback-selection}
\mathbb E_t\int_t^{T\wedge\tau_N}
\left[F_s^{\nu_s^{\varepsilon,t,x}}V_s(X_s^{\varepsilon,t,x})-\lambda_s\right]ds
\leq\varepsilon.
\end{equation}

By the It\^o--Wentzell formula, the coercivity of $S$,
\eqref{eq:growth}, and \eqref{eq:stopped-feedback-selection}, we obtain,
\begin{align}
\mathbb E_t\int_t^{T\wedge\tau_N}
\left(|X_s^{\varepsilon,t,x}|^2+|\nu_s^{\varepsilon,t,x}|^2\right)ds
&\leq C\mathbb E_t\int_t^{T\wedge\tau_N}
\ell_s(X_s^{\varepsilon,t,x},\nu_s^{\varepsilon,t,x})\,ds
\nonumber\\
&\hspace{-25mm}=C\left[
V_t(x)-\mathbb E_tV_{T\wedge\tau_N}
(X_{T\wedge\tau_N}^{\varepsilon,t,x})
+\mathbb E_t\int_t^{T\wedge\tau_N}
\!\!\!\!F_s^{\nu_s^{\varepsilon,t,x}}V_s(X_s^{\varepsilon,t,x})\,ds
\right]
\label{eq:stopped-state-control-bound}\\
&\hspace{-25mm}
\leq C\left[
V_t(x)+C+\mathbb E_t\int_t^T|\lambda_s|\,ds+\varepsilon
\right]
\leq C_{T,x,\varepsilon},
\nonumber
\end{align}
uniformly in $N$, where the last inequality follows from
$\lambda\in\mathbb L^1_{\mathrm{unif}}(\mathbb R)$ and \eqref{eq:growth}.

By the boundedness of the coefficients, the
Burkholder--Davis--Gundy inequality, and
\eqref{eq:stopped-state-control-bound}, we obtain
$\sup_{N\geq1}\mathbb E_t\sup_{t\leq s\leq T\wedge\tau_N}|X_s^{\varepsilon,t,x}|^2\leq C_{T,x,\varepsilon}$.
Consequently, by the conditional Markov inequality,
$\mathbb E_t\mathbf 1_{\{\tau_N\leq T\}}\leq C_{T,x,\varepsilon}/N^2$.
Letting $N\to\infty$ gives $\mathbb E_t\mathbf 1_{\{\zeta\leq T\}}=0$.
Since $T>t$ is arbitrary, the closed-loop solution is global.

Letting $N\to\infty$ in \eqref{eq:stopped-state-control-bound}
and using Fatou's lemma shows that $\nu^{\varepsilon,t,x}\in\mathcal U$.
Finally, applying \eqref{eq:approximate-feedback-selection} to
$Y=X^{\varepsilon,t,x}$ yields \eqref{eq:verification-control-bound}.
\end{proof}

\begin{proof}[Proof of Theorem~\ref{thm:general-decreasing-kernels}]
Let $(V,\lambda)$ be any solution in the sense of
Definition~\ref{def:ergodic_solution}.
Fix $K$ and $x$.
By Lemma~\ref{lem:existence-verification-controls},
$\mathcal U_{\mathrm{ver}}^{V,\lambda}(x)\neq\emptyset$.
Set
$\Lambda(T):=\mathbb E\int_0^T\lambda_s\,ds$ and
$G^\nu(T):=J^T(x;\nu)-\Lambda(T)$. Throughout the proof, $C_x$ denotes a generic finite constant depending
on the fixed initial state $x$, whose value may change from one
occurrence to another. The arguments below are understood
after localization whenever needed.

For every $\nu\in\mathcal U$, \eqref{eq:stochastic-ergodic-hjb} and
\eqref{eq:growth}, together with the It\^o--Wentzell formula, yield
\[
G^\nu(T)
\geq
\mathbb EV_0(x)-\mathbb EV_T(X_T^\nu)
\geq
-C_x-C\mathbb E|X_T^\nu|^2,
\]
where $C$ is the constant in \eqref{eq:growth}. Moreover, $G^\nu$ is absolutely continuous with
\[
(G^\nu)'(t)
=\mathbb E\left[S_tX_t^\nu\cdot X_t^\nu+|\nu_t|^2-\lambda_t\right]
\geq\mathbb E\left[\delta|X_t^\nu|^2-\lambda_t\right],
\quad dt\text{-a.e.},
\]
since $S_t\geq\delta I_n$, $d\mathbb P\otimes dt$-a.e. Combining the preceding
two inequalities, we obtain
$
(G^\nu)'(t)+\frac{\delta}{C}G^\nu(t)
\geq -C_x-\mathbb E|\lambda_t|.
$
Since $\lambda\in\mathbb L^1_{\mathrm{unif}}(\mathbb R)$ and
$G^\nu(0)=0$, integrating this differential inequality yields
\begin{equation}\label{eq:finite-lower-bound}
G^\nu(T)\geq
-\int_0^T e^{-\frac{\delta}{C}(T-s)}
\left[C_x+\mathbb E|\lambda_s|\right]ds
\geq-C_x,
\qquad \nu\in\mathcal U,\quad T\geq0.
\end{equation}

Now let $\nu\in\mathcal U_{\mathrm{ver}}^{V,\lambda}(x)$ and set
$r^\nu(s):=\mathbb E[F_s^{\nu_s}V_s(X_s^\nu)-\lambda_s], R^\nu(T):=\int_0^T r^\nu(s)\,ds$.
Then $r^\nu\geq0$ and $R^\nu(T)/T\to0$.
By the It\^o--Wentzell formula and the lower bound in \eqref{eq:growth}, we see that
\begin{equation}\label{eq:finite-verification-upper-bound}
G^\nu(T)=\mathbb EV_0(x)-\mathbb EV_T(X_T^\nu)+R^\nu(T)
\leq C_x+R^\nu(T).
\end{equation}
Since $G^\nu(0)=0$ and $K_\rho$ is nonincreasing,
integration by parts and \eqref{eq:finite-lower-bound} provide
\begin{align}
\int_0^T K_\rho(s)\,dG^\nu(s)
&=
K_\rho(T)G^\nu(T)
-\int_{(0,T]}G^\nu(s)\,dK_\rho(s)
\label{eq:kernel-integration-by-parts}\\
&\geq
-C_xK_\rho(T)
-C_x\bigl(K_\rho(0)-K_\rho(T)\bigr)
=-C_x\rho K(0),
~T>0.
\notag
\end{align}
Letting $T\to\infty$ yields, for every $\nu\in\mathcal U$,
\begin{equation}\label{eq:kernel-universal-lower-bound}
J^{K_\rho}(x;\nu)-\Lambda^{K_\rho}
\geq
-C_x\rho K(0).
\end{equation}
On the other hand, combining \eqref{eq:finite-verification-upper-bound} and
\eqref{eq:kernel-integration-by-parts} gives
\begin{equation}\label{eq:kernel-verification-upper-bound-finite}
\int_0^T K_\rho(s)\,dG^\nu(s)
\leq
K_\rho(T)G^\nu(T)
+C_x\bigl(K_\rho(0)-K_\rho(T)\bigr)
-\int_{(0,T]}R^\nu(s)\,dK_\rho(s).
\end{equation}
Since $K_\rho\in\mathcal K$ is nonincreasing and integrable,
$\lim_{T\to\infty}TK_\rho(T)=0$ for every $\rho>0$. Moreover,
\eqref{eq:finite-lower-bound}--\eqref{eq:finite-verification-upper-bound}
and $R^\nu(T)/T\to0$ yield $G^\nu(T)/T\to0$. Hence
$K_\rho(T)G^\nu(T)\to0$. Letting $T\to\infty$ in
\eqref{eq:kernel-verification-upper-bound-finite} therefore yields
\[
J^{K_\rho}(x;\nu)-\Lambda^{K_\rho}
\leq
C_x\rho K(0)
-\int_{(0,\infty)}R^\nu(s)\,dK_\rho(s)
=
C_x\rho K(0)
+\int_0^\infty K_\rho(s)r^\nu(s)\,ds.
\]
Since $r^\nu\geq0$ and $R^\nu(T)/T\to0$,
Lemma~\ref{lemma:upper-sandwich} gives
$\int_0^\infty K_\rho(s)r^\nu(s)\,ds\to0$.
Together with \eqref{eq:kernel-universal-lower-bound}, this yields
$J^{K_\rho}(x;\nu)=\Lambda^{K_\rho}+o(1)$ for all
$\nu\in\mathcal U_{\mathrm{ver}}^{V,\lambda}(x)$.
Moreover, by \eqref{eq:kernel-universal-lower-bound},
\[
\liminf_{\rho\downarrow0}J^{K_\rho}(x;\widetilde\nu)
\geq\lambda_*(K),
\qquad
\limsup_{\rho\downarrow0}J^{K_\rho}(x;\widetilde\nu)
\geq\lambda^*(K).
\]
For all $\nu\in\mathcal U_{\mathrm{ver}}^{V,\lambda}(x)$, we have
$
\Lambda^{K_\rho}-C_x\rho K(0)
\leq\inf_{\widetilde\nu\in\mathcal U}
J^{K_\rho}(x;\widetilde\nu)
\leq J^{K_\rho}(x;\nu)
=\Lambda^{K_\rho}+o(1).
$
Taking lower and upper limits proves
\eqref{eq:kernel-lower-value}--\eqref{eq:kernel-upper-value}
and the optimality assertion in \textup{(i)}.

For \textup{(ii)}, let $(V^0,\lambda^0)$ be defined by
\eqref{eq:quadratic-ergodic-value-field}.
Substituting $u_t^*(x)$ into \eqref{eq:completion-of-squares} leads to
\begin{equation}\label{eq:affine-feedback-minimizer}
\ell_t(x,u_t^*(x))+\mathcal L_t^{u_t^*(x)}V_t^0(x)=\lambda_t^0,
\quad\text{for }x\in\mathbb R^n,
\quad d\mathbb P\otimes dt\text{-a.e.}
\end{equation}
The estimates in
Proposition~\ref{prop:infinite-horizon-riccati} \textup{(i)--(ii)}
ensure the local integrability of the drift coefficients and
the local square integrability of the diffusion coefficients.
Since the closed-loop equation is linear, it admits a unique
global strong solution $X^{x,*}$.

By \eqref{eq:affine-feedback-minimizer}, the proof of
Lemma~\ref{lem:existence-verification-controls} applies with
$\varepsilon=0$ and yields $\nu^{x,*}\in\mathcal U$.
Moreover, \eqref{eq:affine-feedback-minimizer} implies
$\nu^{x,*}\in\mathcal U_{\mathrm{ver}}^{V^0,\lambda^0}(x)$.
Part \textup{(i)} therefore shows
that $\nu^{x,*}$ attains the infima in \eqref{eq:kernel-lower-value}
and \eqref{eq:kernel-upper-value}
for every $K\in\mathcal K$.
\end{proof}

\section{Well-posedness of the LQ stochastic ergodic HJB equation}
\label{sec:finite-horizon}

In this section, we first recall the well-posedness results of
\cite{KohlmannTang2003,Tang2003} for the finite-horizon backward
stochastic Riccati equation and the associated dual BSDE.
We then prove Proposition~\ref{prop:infinite-horizon-riccati},
building on the existence results and methods of
\cite{GuatteriMasiero2010} to establish well-posedness of the
corresponding infinite-horizon equations, together with uniform
coercivity of the Riccati solution and uniform local $L^2$
estimates for the martingale integrands. We use the notation $N$, $L$, $\Gamma$, $H$, $M^i$, $\varphi$,
$\eta$, $\mathcal R$, and $\mathcal G$ introduced in
\eqref{eq:LQ-coefficient-maps}.

\subsection{Finite-horizon value field and equations}\label{sec:finite_horizon}

We first recall from \cite{KohlmannTang2003,Tang2003}
the quadratic--affine representation of the value field associated
with the finite-horizon LQ control problem. For later
use, we allow a quadratic terminal cost 
$G_T\in L^\infty(\Omega,\mathcal F_T;\mathcal S_+^n)$, and we define for $ 0\le t\le T, x\in\mathbb R^n:$
\begin{equation*}
J_t^{T,G}(x;\nu)
:=J_t^T(x;\nu)+\mathbb E_t\left[G_TX_T^{t,x;\nu}\cdot X_T^{t,x;\nu}\right],
V_t^{T,G}(x)
:=\operatorname*{ess\,inf}_{\nu\in\mathcal U}J_t^{T,G}(x;\nu).
\end{equation*}
The terminal-cost-free problem introduced in
\eqref{eq:horizon-cost} is recovered by taking $G_T=0$,
i.e., $V_t^T(x):=V_t^{T,0}(x)$. Since the dynamics are affine and the running cost is quadratic, the LQ
structure naturally suggests the quadratic--affine ansatz
\begin{equation}\label{eq:finite-horizon-value-representation}
    V_t^{T,G}(x)
    =
    P_t^T x\cdot x
    +
    2r_t^T\cdot x
    +
    \gamma_t^T,
    \qquad x\in\mathbb R^n.
\end{equation}
Here $P^T$ is $\mathcal S_+^n$-valued, $r^T$ is
$\mathbb R^n$-valued, and $\gamma^T$ is real-valued.
For simplicity, we omit their dependence on the terminal data
$G_T$ from the notation. 

The coefficients $P^T$ and $r^T$ solve the finite-horizon backward
stochastic Riccati equation and the associated dual BSDE, respectively:
\begin{align}
dP_t^T
&=-\mathcal R_t(P_t^T,Q_t^T)\,dt
+\sum_{i=1}^dQ_t^{i,T}\,dW_t^i,
\qquad P_T^T=G_T,
\label{eq:finite-riccati}\\
dr_t^T
&=-\mathcal G_t(P_t^T,Q_t^T,r_t^T,k_t^T)\,dt
+\sum_{i=1}^dk_t^{i,T}\,dW_t^i,
\qquad r_T^T=0.
\label{eq:finite_dual}
\end{align}

With $N_t^T:=N_t(P_t^T)$ and
$\eta_t^T:=\eta_t(P_t^T,r_t^T,k_t^T)$,
the coefficient $\gamma^T$ is given by
\begin{equation}\label{eq:finite-horizon-lambda}
\gamma_t^T:=\mathbb E_t\int_t^T\lambda_s^T\,ds,
\;
\lambda_t^T
:=2f_t\cdot r_t^T
+\sum_{i=1}^d\left(P_t^Tg_t^i\cdot g_t^i+2k_t^{i,T}\cdot g_t^i\right)
-(N_t^T)^{-1}\eta_t^T\cdot\eta_t^T.
\end{equation}

\begin{proposition}[\cite{KohlmannTang2003,Tang2003}]
\label{prop:bsde_finite_horizon}
Let Assumption~\ref{ass:bounded-coefficients} hold.
Fix $T>0$, and let $G_T\in L^\infty(\Omega,\mathcal F_T;\mathcal S_+^n)$.
Then equations~\eqref{eq:finite-riccati} and \eqref{eq:finite_dual}
admit unique solutions
$(P^T,Q^T)\in\mathbb S_T^\infty(\mathcal S_+^n)
\times\mathbb H_T^2((\mathcal S^n)^d)$ and
$(r^T,k^T)\in\mathbb S_T^\infty(\mathbb R^n)
\times\mathbb H_T^2((\mathbb R^n)^d)$, respectively.
Moreover, these solutions, together with $\gamma^T$ defined in
\eqref{eq:finite-horizon-lambda}, yield the representation
\eqref{eq:finite-horizon-value-representation} of the value field
$V^{T,G}$.
\end{proposition}

The existence and uniqueness of a solution to the backward stochastic
Riccati equation \eqref{eq:finite-riccati} follow from \cite{Tang2003}.
Given this solution, the existence and uniqueness of a solution to the
dual BSDE \eqref{eq:finite_dual} follow from \cite{KohlmannTang2003}.
Notice that the coefficient of the quadratic control term in
\cite{KohlmannTang2003} equals $I_k$ here, which leads to
$N_t^T\geq I_k$ in our setting.
Their results therefore apply without any nondegeneracy condition
on $D$, as explained in the concluding remarks of Section~4 of
\cite{KohlmannTang2003}.
Although \cite{KohlmannTang2003} establishes
$\mathbb E[\sup_{0\leq t\leq T}|r_t^T|^2]<\infty$,
it does not assert that $r^T$ is bounded.
In our setting, comparison with the zero control
gives $0\leq V_t^{T,G}(x)\leq C_T(1+|x|^2)$.
Evaluating \eqref{eq:finite-horizon-value-representation} at the
coordinate vectors $\pm e_j$ gives
$4r_t^T\cdot e_j=V_t^{T,G}(e_j)-V_t^{T,G}(-e_j)$,
which shows that $r^T$ is bounded.
The BSDE dynamics imply that $P^T$ and $r^T$ have continuous versions,
so both belong to the stated $\mathbb S_T^\infty$ spaces.

\subsection{Well-posedness of the infinite-horizon Riccati and dual equations}
\label{sec:existence_ergodic_HJB}
In this subsection, we prove the well-posedness and additional
properties stated in Proposition~\ref{prop:infinite-horizon-riccati}.
We rely on the existence of a minimal nonnegative solution to the
infinite-horizon backward stochastic Riccati equation established
in \cite{GuatteriMasiero2010}, together with its associated optimal
feedback and closed-loop stability properties.
For the associated infinite-horizon dual equation, we adapt their
existence proof to the present setting.

\begin{lemma}[\cite{GuatteriMasiero2010}]
\label{lem:homogeneous-closed-loop}
Let Assumptions~\ref{ass:bounded-coefficients} and
\ref{ass:uniform-stabilizability} hold.
Equation~\eqref{eq:infinite-horizon-riccati} admits a minimal
nonnegative solution
$(P,Q)\in\mathbb S^\infty(\mathcal S_+^n)
\times\mathbb H_\infty^2((\mathcal S^n)^d)$.
Here minimality means that $P_t\leq\widetilde P_t$,
$\mathbb P$-a.s. for every $t\geq0$, for any other nonnegative
solution $(\widetilde P,\widetilde Q)$ of
\eqref{eq:infinite-horizon-riccati}.
For every $t\geq0$ and $x\in\mathbb R^n$,
the homogeneous closed-loop equation
\begin{equation}\label{eq:homogeneous-closed-loop}
d\overline X_s^{t,x,*}
=H_s(P_s,Q_s)\overline X_s^{t,x,*}\,ds
+\sum_{i=1}^dM_s^i(P_s,Q_s)\overline X_s^{t,x,*}\,dW_s^i,
\qquad s\geq t,\quad \overline X_t^{t,x,*}=x,
\end{equation}
admits a unique global strong solution $\overline X^{t,x,*}$.
There exist constants $C,\beta>0$, independent of $t$ and $x$, such that
\begin{equation}\label{eq:exponential_stability}
\mathbb E_t|\overline X_s^{t,x,*}|^2
\leq Ce^{-\beta(s-t)}|x|^2,
\qquad s\geq t,\quad\mathbb P\text{-a.s.}
\end{equation}
The induced control
$\overline\nu_s^{t,x,*}:=\Gamma_s(P_s,Q_s)\overline X_s^{t,x,*}$,
$s\geq t$, belongs to $\mathcal U$ and is optimal for the homogeneous
infinite-horizon cost defined in
\eqref{eq:homogeneous-infinite-horizon-cost}:
\begin{equation}\label{eq:homogeneous-value-representation}
P_tx\cdot x
=\overline J_t^\infty(x;\overline\nu^{t,x,*})
=\operatorname*{ess\,inf}_{\nu\in\mathcal U}
\overline J_t^\infty(x;\nu),
\qquad
\mathbb P\text{-a.s., for all }t\geq0,\ x\in\mathbb R^n.
\end{equation}
\end{lemma}

The existence of the minimal nonnegative solution and the
optimal feedback and value representation follow from
\cite[Corollary~3.7 and Theorem~3.8]{GuatteriMasiero2010}.
Assumption~\ref{ass:uniform-stabilizability} and
\eqref{eq:homogeneous-value-representation} imply that $P$ is
uniformly bounded and satisfies the stability condition in
\cite[Definition~4.1]{GuatteriMasiero2010}.
Thus, \eqref{eq:exponential_stability} follows from
\cite[Theorem~4.3]{GuatteriMasiero2010}.

In the following, for simplicity, we write
$N_t:=N_t(P_t)$ and $L_t:=L_t(P_t,Q_t)$,
and similarly for $\Gamma_t$, $H_t$, $M_t^i$, and $\varphi_t$,
with $(P,Q)$ given by Lemma~\ref{lem:homogeneous-closed-loop}.

\begin{proof}[Proof of Proposition~\ref{prop:infinite-horizon-riccati}]
Throughout the proof, $C$ denotes a generic positive constant whose
value may change from one occurrence to another.
All applications of It\^o's formula below are understood after
localization when needed.

\textup{(i)}
Let $(P,Q)$ be the minimal nonnegative solution given by
Lemma~\ref{lem:homogeneous-closed-loop}.
We first prove the uniform coercivity of $P$.
By Assumption~\ref{ass:uniform-stabilizability},
$S_t\geq\delta I_n$, $d\mathbb P\otimes dt$-a.e.
The dynamics \eqref{eq:homogeneous-controlled-dynamics}, the boundedness
of the coefficients, the Burkholder--Davis--Gundy
inequality, and Gronwall's inequality yield
\begin{equation}\label{eq:homogeneous-short-time-increment}
\mathbb E_t|\overline X_{t+h}^{t,x;\nu}-x|^2
\leq Ce^{Ch}\left(
h|x|^2+\mathbb E_t\int_t^{t+h}|\nu_r|^2\,dr
\right),
\qquad h>0,
\end{equation}
where $C$ is independent of $t$, $x$, $\nu$, and $h$.
Using $|y|^2\geq\frac12|x|^2-|y-x|^2$ and integrating
\eqref{eq:homogeneous-short-time-increment}, choose $h>0$
sufficiently small that $Che^{Ch}\leq1/4$. Then
\[
\overline J_t^\infty(x;\nu)
\geq\min\{\delta,1\}\mathbb E_t\int_t^{t+h}
\left(|\overline X_s^{t,x;\nu}|^2+|\nu_s|^2\right)ds
\geq\frac{\min\{\delta,1\}h}{4}|x|^2.
\]
Using
\eqref{eq:homogeneous-value-representation} gives
$P_t\geq\delta_P I_n$, with
$\delta_P:=\min\{\delta,1\}h/4>0$.

Next, since $P\geq0$ and $L_t^\top N_t^{-1}L_t\geq0$,
the boundedness of $P$ and the coefficients gives
$\operatorname{Tr}[P_t\mathcal R_t(P_t,Q_t)]\leq C(1+|Q_t|)$. Applying It\^o's formula to $|P|^2$ on $[t,t+1]$ and using Young's
inequality, we obtain, for a fixed constant $C>0$ independent of $t$,
\[
\mathbb E_t\int_t^{t+1}|Q_s|^2\,ds
\leq C+C\mathbb E_t\int_t^{t+1}|Q_s|\,ds
\leq C+\frac{C^2}{2}
+\frac12\mathbb E_t\int_t^{t+1}|Q_s|^2\,ds.
\]

Thus,
\begin{equation}\label{eq:infinite-riccati-Q-uniform-bound}
\|Q\|_{\mathbb L^2_{\mathrm{unif}}((\mathcal S^n)^d)}^2
=
\sup_{t\geq0}
\left\|\mathbb E_t\int_t^{t+1}|Q_s|^2\,ds\right\|_\infty
\leq 2C+C^2.
\end{equation}

To prove uniqueness, let $(\widetilde P,\widetilde Q)$ be another
solution in the stated class, and set
$\widetilde N_t:=N_t(\widetilde P_t)$ and
$\widetilde\Gamma_t:=-\widetilde N_t^{-1}
L_t(\widetilde P_t,\widetilde Q_t)$.
Applying It\^o's formula to
$(P_s-\widetilde P_s)\overline X_s^{t,x,*}\cdot\overline X_s^{t,x,*}$
and completing the square gives

\begin{equation}\label{eq:infinite-riccati-uniqueness-identity}
(P_t-\widetilde P_t)x\cdot x
\!=\!\mathbb E_t\!\left[
(P_T-\widetilde P_T)\overline X_T^{t,x,*}
\cdot\overline X_T^{t,x,*}\right]
\!+\!\mathbb E_t\!\int_t^T
\left|\widetilde N_s^{1/2}\!
(\Gamma_s-\widetilde\Gamma_s)\overline X_s^{t,x,*}\right|^2\,ds.
\end{equation}
By the boundedness of $P$ and $\widetilde P$ and
\eqref{eq:exponential_stability}, the terminal term in
\eqref{eq:infinite-riccati-uniqueness-identity} converges to zero as
$T\to\infty$. Hence $P_t\geq\widetilde P_t$.
Minimality gives the reverse inequality, so $P=\widetilde P$.
Uniqueness of the semimartingale decomposition then gives
$Q=\widetilde Q$, $d\mathbb P\otimes dt$-a.e.

\textup{(ii)}
Since $P$ and the coefficients are bounded and $N_t\geq I_k$,
\begin{equation}\label{eq:dual-coefficient-bound}
|\Gamma_t|^2+|H_t|^2+\sum_{i=1}^d|M_t^i|^2+|\varphi_t|^2
\leq C(1+|Q_t|^2).
\end{equation}
Recall that
$\mathcal G_t(P_t,Q_t,r,k)
=H_t^\top r+\sum_{i=1}^d(M_t^i)^\top k^i+\varphi_t$.
For every $T>0$, consider the finite-horizon equation
\begin{equation}\label{eq:dual-finite-horizon-approximation}
d\widehat r_t^T
=-\left[H_t^\top\widehat r_t^T
+\sum_{i=1}^d(M_t^i)^\top\widehat k_t^{i,T}+\varphi_t\right]dt
+\sum_{i=1}^d\widehat k_t^{i,T}\,dW_t^i,
\qquad \widehat r_T^T=0.
\end{equation}
The restriction of $(P,Q)$ to $[0,T]$ solves
\eqref{eq:finite-riccati} with terminal condition $G_T=P_T$.
Thus, Proposition~\ref{prop:bsde_finite_horizon} gives a unique solution
$\widehat r^T\in\mathbb S_T^\infty(\mathbb R^n)$ and
$\widehat k^T\in\mathbb H_T^2((\mathbb R^n)^d)$ to
\eqref{eq:dual-finite-horizon-approximation}.

Applying It\^o's formula to
$\widehat r_s^T\cdot\overline X_s^{t,x,*}$, using the dynamics
\eqref{eq:homogeneous-closed-loop} and
\eqref{eq:dual-finite-horizon-approximation},
and then applying the tower property and the conditional
Cauchy--Schwarz inequality, we obtain, for every $t\leq s_0\leq T$,
\begin{align}
\left|\mathbb E_t\left[
\widehat r_{s_0}^T\cdot\overline X_{s_0}^{t,x,*}\right]\right|
&=\Big|\mathbb E_t\int_{s_0}^T
\varphi_s\cdot\overline X_s^{t,x,*}\,ds\Big|
\notag\\
&\leq\sum_{n=0}^\infty
\left(\mathbb E_t\int_{s_0+n}^{s_0+n+1}|\varphi_s|^2\,ds\right)^{1/2}
\left(\mathbb E_t\int_{s_0+n}^{s_0+n+1}
|\overline X_s^{t,x,*}|^2\,ds\right)^{1/2}
\label{eq:dual-representation-tail-bound}\\
&\leq C|x|\sum_{n=0}^\infty e^{-\beta(s_0+n-t)/2}
\leq Ce^{-\beta(s_0-t)/2}|x|.
\notag
\end{align}
where the last line follows from
\eqref{eq:exponential_stability},\eqref{eq:infinite-riccati-Q-uniform-bound}, and
\eqref{eq:dual-coefficient-bound}.
In particular, taking $s_0=t$ gives $|\widehat r_t^T|\leq C$.

Fix $T_2\geq T_1\geq T_0>0$, and set
$\Delta r:=\widehat r^{T_2}-\widehat r^{T_1}$ and
$\Delta k:=\widehat k^{T_2}-\widehat k^{T_1}$.
Applying It\^o's formula to
$\Delta r_s\cdot\overline X_s^{t,x,*}$ on $[t,T_1]$ and using
\eqref{eq:dual-representation-tail-bound} gives
\[
|\Delta r_t\cdot x|
=\left|\mathbb E_t\left[
\widehat r_{T_1}^{T_2}\cdot\overline X_{T_1}^{t,x,*}\right]\right|
\leq Ce^{-\beta(T_1-t)/2}|x|,
\qquad 0\leq t\leq T_1.
\]
Hence,
\begin{equation}\label{eq:dual-r-approximation-convergence}
\left\|\sup_{0\leq t\leq T_0}|\Delta r_t|\right\|_\infty
\leq Ce^{-\beta(T_1-T_0)/2}.
\end{equation}

Together with $|\widehat r_t^T|\leq C$, this shows that
$\widehat r^T$ converges uniformly on compact time intervals,
in $L^\infty(\Omega)$, to a bounded continuous progressively
measurable process $r$.

Subtracting \eqref{eq:dual-finite-horizon-approximation} for the two
horizons, applying It\^o's formula to $|\Delta r|^2$, and using
Young's inequality, we obtain
\begin{equation}\label{eq:dual-k-approximation-convergence}
\mathbb E\int_0^{T_0}|\Delta k_s|^2\,ds
\leq C\mathbb E|\Delta r_{T_0}|^2
+C\mathbb E\int_0^{T_0}
\left(|H_s|+\sum_{i=1}^d|M_s^i|^2\right)|\Delta r_s|^2\,ds
\leq Ce^{-\beta(T_1-T_0)},
\end{equation}
where the last inequality follows from
\eqref{eq:infinite-riccati-Q-uniform-bound},
\eqref{eq:dual-coefficient-bound}, and
\eqref{eq:dual-r-approximation-convergence}.
Thus, $\widehat k^T$ converges in
$\mathbb H_{T_0}^2((\mathbb R^n)^d)$ to a process $k$.
Since $T_0$ is arbitrary, these limits define
$k\in\mathbb H_\infty^2((\mathbb R^n)^d)$.

For $0\leq t\leq T_0\leq T$, write
\eqref{eq:dual-finite-horizon-approximation} in integral form on
$[t,T_0]$. The convergence in
\eqref{eq:dual-r-approximation-convergence} and
\eqref{eq:dual-k-approximation-convergence}, together with the local
integrability of $H$ and the local square integrability of $M$,
allows us to pass to the limit.
Consequently, $(r,k)$ solves \eqref{eq:infinite_dual}.

Applying It\^o's formula to $|r|^2$ on $[t,t+1]$ and using the
boundedness of $r$ and Young's inequality, we obtain
\[
\frac12\mathbb E_t\int_t^{t+1}|k_s|^2\,ds
\leq C+C\mathbb E_t\int_t^{t+1}
\left(|H_s|+|\varphi_s|+\sum_{i=1}^d|M_s^i|^2\right)ds
\leq C',
\]
where $C'$ is independent of $t$ by
\eqref{eq:infinite-riccati-Q-uniform-bound} and \eqref{eq:dual-coefficient-bound}.
Hence $k\in\mathbb L^2_{\mathrm{unif}}((\mathbb R^n)^d)$.

Finally, let $(\widetilde r,\widetilde k)$ be another solution in the
stated class. Applying It\^o's formula to
$(r_s-\widetilde r_s)\cdot\overline X_s^{t,x,*}$ and using the
boundedness of $r$ and $\widetilde r$ together with
\eqref{eq:exponential_stability}, we obtain
\[
|(r_t-\widetilde r_t)\cdot x|
=\left|\mathbb E_t\left[
(r_T-\widetilde r_T)\cdot\overline X_T^{t,x,*}\right]\right|
\leq Ce^{-\beta(T-t)/2}|x|,
\qquad T>t.
\]
Letting $T\to\infty$ and using the arbitrariness of $x$ gives
$r=\widetilde r$.
Uniqueness of the semimartingale decomposition then gives
$k=\widetilde k$, $d\mathbb P\otimes dt$-a.e.
\end{proof}

\begin{remark}
Let $g=0$ and
$\mathbb E\int_0^\infty|f_t|^2\,dt<\infty$.
Then, by the infinite-horizon result of \cite{GuatteriMasiero2010},
we have
\[
\operatorname*{ess\,inf}_{\nu\in\mathcal U}J_t^\infty(x;\nu)
=P_tx\cdot x+2r_t\cdot x+\gamma_t,
\qquad
\gamma_t:=\mathbb E_t\int_t^\infty\lambda_s^0\,ds,
\]
where $\lambda^0$ is defined in
\eqref{eq:quadratic-ergodic-value-field}, with $g=0$.
When $f$ is only assumed to be bounded, however,
$\lambda^0$ need not be integrable over the infinite horizon, so
$\mathbb E_t\int_t^\infty\lambda_s^0\,ds$ need not be finite.
For this reason, the construction of $(V^0,\lambda^0)$ in
\eqref{eq:quadratic-ergodic-value-field}
does not involve this infinite-horizon value representation.
\end{remark}

\subsection{Uniqueness of solutions to the stochastic ergodic HJB equation}
\label{sec:proof_ergodic_HJB}

We now prove the uniqueness results in
Theorem~\ref{thm:existence_ergodic_hjb}~\textup{(ii)--(iii)}.
We first derive a stability property of the approximately optimal
controls constructed in
Lemma~\ref{lem:existence-verification-controls}.

\begin{lemma}
\label{lem:approximate-feedback-stability}

Let Assumptions~\ref{ass:bounded-coefficients} and
\ref{ass:uniform-stabilizability} hold, and let $(V,\lambda)$ be a solution of the stochastic ergodic HJB equation
in the sense of Definition~\ref{def:ergodic_solution} such that
\[
V_s(x)\geq c_0|x|^2-C_0,
\quad\mathbb P\text{-a.s., for all }s\geq0,\ x\in\mathbb R^n,
\quad\text{for some constants }c_0,C_0>0.
\]
For $t\geq0$, $\varepsilon\in(0,1)$, and $x\in\mathbb R^n$, let
$\nu^{\varepsilon,t,x}$ be a control given by
Lemma~\ref{lem:existence-verification-controls} with error
$\varepsilon$, and denote the corresponding state process by
$X^{\varepsilon,t,x}$.
Then there exist constants $C,\beta>0$, independent of $t$, $x$,
and $\varepsilon$, such that $\mathbb P\text{-a.s.}$ for $x,y\in\mathbb R^n$
\begin{equation}\label{eq:approximate-state-moment}
\mathbb E_t|X_s^{\varepsilon,t,x}|^2
\leq Ce^{-\beta(s-t)}|x|^2+C,\, s\geq t,
\,
\lim_{s\to\infty}\mathbb E_t
\left|X_s^{\varepsilon,t,x}-X_s^{\varepsilon,t,y}\right|^2=0.
\end{equation}
\end{lemma}

\begin{proof}
Fix $t\geq0$ and $\varepsilon\in(0,1)$.
Throughout the proof, $C$ denotes a generic positive constant,
independent of $t$, $x$, $y$, and $\varepsilon$, whose value may
change from one occurrence to another.

\emph{Step 1:}
For $x\in\mathbb R^n$, write
$X^x:=X^{\varepsilon,t,x}$ and
$\nu^x:=\nu^{\varepsilon,t,x}$, and set
$
q_x(s):=\mathbb E_t\left[F_s^{\nu_s^x}V_s(X_s^x)-\lambda_s\right].
$
By Lemma~\ref{lem:existence-verification-controls},
$q_x\geq0$ and $\int_t^\infty q_x(s)\,ds\leq\varepsilon$.
Set $y_x(s):=\mathbb E_tV_s(X_s^x)$.
The It\^o--Wentzell formula, the upper quadratic bound in
\eqref{eq:growth}, and the coercivity of $S$ give, for some
$\beta>0$ and a.e. $s\geq t,$
\[
y_x'(s)
=\mathbb E_t\left[\lambda_s-S_sX_s^x\cdot X_s^x-|\nu_s^x|^2\right]+q_x(s)
\leq\mathbb E_t|\lambda_s|+C-\beta y_x(s)+q_x(s).
\]
Integrating this inequality and using the coercive lower bound on $V$
and $V_t(x)\leq C(1+|x|^2)$ from \eqref{eq:growth}, we obtain
\[
\begin{aligned}
c_0\mathbb E_t|X_s^x|^2-C_0
\leq y_x(s)
&\leq e^{-\beta(s-t)}V_t(x)+C
+\int_t^se^{-\beta(s-r)}(\mathbb E_t|\lambda_r|+q_x(r))\,dr\\
&\leq Ce^{-\beta(s-t)}(1+|x|^2)+C+\varepsilon.
\end{aligned}
\]
Here the first integral is uniformly bounded by
$\lambda\in\mathbb L^1_{\mathrm{unif}}(\mathbb R)$, using the tower
property and a unit-interval decomposition, while the second is bounded
by $\varepsilon$.
Since $\varepsilon<1$, this proves the first estimate in
\eqref{eq:approximate-state-moment}.

\emph{Step 2:}
Fix $x,y\in\mathbb R^n$ and set
$X^m:=(X^x+X^y)/2$, $\nu^m:=(\nu^x+\nu^y)/2$,
$Z:=X^x-X^y$, and $w:=\nu^x-\nu^y$.
By the affine structure of the dynamics, $\nu^m\in\mathcal U$ and
$X^m$ is the state starting from $(x+y)/2$ at time $t$ under
the control $\nu^m$.
Define
$
D_s:=\frac12V_s(X_s^x)+\frac12V_s(X_s^y)-V_s(X_s^m).
$
Applying the It\^o--Wentzell formula to the three state-control pairs
and using the quadratic structure of $\ell$, we obtain
\begin{equation}\label{eq:approximate-midpoint-estimate}
\begin{aligned}
&\frac14\mathbb E_t\int_t^T
\left(S_sZ_s\cdot Z_s+|w_s|^2\right)ds
\\=&D_t-\mathbb E_tD_T
+\mathbb E_t\int_t^T
\left[\frac12F_s^{\nu_s^x}V_s(X_s^x)
+\frac12F_s^{\nu_s^y}V_s(X_s^y)
-F_s^{\nu_s^m}V_s(X_s^m)\right]ds\\
\leq &D_t-\mathbb E_tD_T
+\frac12\mathbb E_t\int_t^T
\left[F_s^{\nu_s^x}V_s(X_s^x)-\lambda_s
+F_s^{\nu_s^y}V_s(X_s^y)-\lambda_s\right]ds\\\leq &D_t-\mathbb E_tD_T+\varepsilon.
\end{aligned}
\end{equation}
Here the first inequality follows from
$F_s^{\nu_s^m}V_s(X_s^m)\geq\lambda_s$ by
\eqref{eq:stochastic-ergodic-hjb}, and the second follows from
Lemma~\ref{lem:existence-verification-controls}.
By \eqref{eq:growth}, \eqref{eq:approximate-state-moment}, and
$|X_T^m|^2\leq(|X_T^x|^2+|X_T^y|^2)/2$,
$\mathbb E_tD_T$ is bounded from below uniformly in $T$.
The coercivity of $S$ and \eqref{eq:approximate-midpoint-estimate}
therefore imply
\begin{equation}\label{eq:approximate-difference-integrability}
\mathbb E_t\int_t^\infty
\left(|Z_s|^2+|w_s|^2\right)ds<\infty.
\end{equation}
Moreover,
$
dZ_s=(A_sZ_s+B_sw_s)\,ds
+\sum_{i=1}^d(C_s^iZ_s+D_s^iw_s)\,dW_s^i.
$
Applying It\^o's formula to $|Z|^2$ and using the boundedness of the
coefficients shows that $m(s):=\mathbb E_t|Z_s|^2$ is locally
absolutely continuous and satisfies
$|m'(s)|\leq C\mathbb E_t\left(|Z_s|^2+|w_s|^2\right)$
for almost every $s\geq t$.
By \eqref{eq:approximate-difference-integrability}, both $m$ and
$m'$ are integrable on $(t,\infty)$, almost surely.
Thus, $m$ has a finite limit at infinity, which must be zero since
$m$ is nonnegative and integrable.
This proves the claimed convergence.
\end{proof}

\begin{proof}[Proof of Theorem~\ref{thm:existence_ergodic_hjb} \textup{(ii)}]
Let $(V^0,\lambda^0)$ be the solution in~\eqref{eq:quadratic-ergodic-value-field}.
Throughout the proof, $C$ denotes a generic positive constant whose
value may change from one occurrence to another but is independent
of the initial time.
Let $(V,\lambda)$ be another solution of the
stochastic ergodic HJB equation, and fix $t\geq0$.
For each $(U,\theta)\in\{(V^0,\lambda^0),(V,\lambda)\}$,
choose a control $\nu^\theta$ from
Lemma~\ref{lem:existence-verification-controls}
for the corresponding solution, with initial time $t$,
initial state $0$, and $\varepsilon=1$.
The same arguments as those leading to
\eqref{eq:finite-lower-bound} and
\eqref{eq:finite-verification-upper-bound} give their conditional
versions starting from $(t,0)$.
Lemma~\ref{lem:existence-verification-controls} ensures that
the conditional counterpart of $R^{\nu^\theta}(T)$ in \eqref{eq:finite-verification-upper-bound},
namely
$R_t^{\nu^\theta}(T):=\mathbb E_t\int_t^T
[F_s^{\nu_s^\theta}U_s(X_s^{t,0;\nu^\theta})-\theta_s]\,ds$,
is bounded by $1$. Consequently,
\[
-C
\leq\operatorname*{ess\,inf}_{\nu\in\mathcal U}J_t^T(0;\nu)
-\mathbb E_t\int_t^T\theta_s\,ds
\leq J_t^T(0;\nu^\theta)-\mathbb E_t\int_t^T\theta_s\,ds
\leq C+1,
\qquad T\geq t.
\]
Here $C$ can be chosen independently of $t$ and $T$ and uniformly
over $\theta\in\{\lambda^0,\lambda\}$.
Comparing these inequalities for $\theta=\lambda^0$ and
$\theta=\lambda$ gives
\begin{equation}\label{eq:cumulative-lambda-uniqueness}
\left|\mathbb E_t\int_t^T(\lambda_s^0-\lambda_s)\,ds\right|
\leq 2C+1,
\qquad 0\leq t\leq T.
\end{equation}
\end{proof}
\begin{proof}[Proof of Theorem~\ref{thm:existence_ergodic_hjb} \textup{(iii)}]
Assume that $V$ satisfies the coercivity and uniform local
continuity conditions in \eqref{eq:uniform-local-continuity}.
Since $P\geq\delta_P I_n$ and $r$ is bounded, the constructed solution
$V^0$ is coercive. Hence
Lemma~\ref{lem:approximate-feedback-stability} applies to both
solutions. Fix $(t,x)$ and choose a control given by
Lemma~\ref{lem:existence-verification-controls} for
$(V,\lambda)$ with initial condition $(t,x)$
and $\varepsilon=1/2$. Denote the corresponding state by $X$.
The It\^o--Wentzell formula, \eqref{eq:growth},
Lemma~\ref{lem:approximate-feedback-stability}, and
\eqref{eq:cumulative-lambda-uniqueness} give
\[
V_t^0(x)-V_t(x)
\leq\mathbb E_t\left[V_T^0(X_T)-V_T(X_T)\right]
+\mathbb E_t\int_t^T(\lambda_s-\lambda_s^0)\,ds+\frac12
\leq C(1+e^{-\beta(T-t)}|x|^2).
\]
Letting $T\to\infty$ and interchanging the two solutions yield
\begin{equation}\label{eq:value-fields-uniform-difference}
|V_t^0(x)-V_t(x)|\leq C,
\quad\mathbb P\text{-a.s.}
\end{equation}
By continuity of $V^0$ and $V$ in $t$ and $x$, this bound
holds $\mathbb P$-a.s. for all $t\geq0$ and $x\in\mathbb R^n$. Fix $t\geq0$, $x,y\in\mathbb R^n$, and $\varepsilon\in(0,1)$.
Apply Lemma~\ref{lem:existence-verification-controls} to
$(V,\lambda)$ with this $\varepsilon$
and initial conditions $(t,x)$ and $(t,y)$ to obtain
controls $\nu^x$ and $\nu^y$, respectively.
Let $X^x,X^y$ be the corresponding states and define
$X^m:=(X^x+X^y)/2$.
Set $G:=V-V^0$. Using \eqref{eq:quadratic-affine-ergodic-solution}, for $T\geq t$, we define
\begin{align}
\widetilde D_T
:&=
\frac12V_T(X_T^x)
+\frac12V_T(X_T^y)
-V_T(X_T^m)\notag\\
&=
\frac14P_T(X_T^x-X_T^y)\cdot(X_T^x-X_T^y)
+\frac12G_T(X_T^x)
+\frac12G_T(X_T^y)
-G_T(X_T^m).
\label{eq:midpoint-gap-decomposition}
\end{align}

For $R,\eta>0$, \eqref{eq:uniform-local-continuity} gives
$\varpi_V^R(\eta)\to0$ a.s. as $\eta\downarrow0$,
while \eqref{eq:growth} gives
$\varpi_V^R(\eta)\leq2C(1+R^2)$.

Set $A_T:=\{|X_T^x|\leq R,\ |X_T^y|\leq R,\
|X_T^x-X_T^y|\leq\eta\}$.
The conditional Markov inequality gives
$\mathbb E_t\mathbf 1_{A_T^c}
\leq R^{-2}\mathbb E_t[|X_T^x|^2+|X_T^y|^2]
+\eta^{-2}\mathbb E_t|X_T^x-X_T^y|^2$.
Together with \eqref{eq:value-fields-uniform-difference} and
\eqref{eq:midpoint-gap-decomposition}, this yields
\begin{align*}
\mathbb E_t|\widetilde D_T|
&\leq
\mathbb E_t\bigl[\varpi_V^R(\eta)\mathbf 1_{A_T}\bigr]
+C\mathbb E_t\bigl[|X_T^x-X_T^y|^2\mathbf 1_{A_T^c}\bigr]
+C\mathbb E_t\mathbf 1_{A_T^c}\\
&\leq
\mathbb E_t\varpi_V^R(\eta)
+C(1+\eta^{-2})\mathbb E_t|X_T^x-X_T^y|^2
+\frac{C}{R^2}\mathbb E_t\bigl[|X_T^x|^2+|X_T^y|^2\bigr].
\end{align*}
By Lemma~\ref{lem:approximate-feedback-stability}, letting
$T\to\infty$ yields
\[
\limsup_{T\to\infty}\mathbb E_t|\widetilde D_T|
\leq\mathbb E_t\varpi_V^R(\eta)+\frac{C}{R^2}.
\]
Letting $\eta\downarrow0$ by dominated convergence
and then $R\to\infty$ yields
$\mathbb E_t|\widetilde D_T|\to0$.

Set $w:=\nu^x-\nu^y$.
Then $X^x-X^y=\overline X^{t,x-y;w}$ satisfies
\eqref{eq:homogeneous-controlled-dynamics} with control $w$.
Applying \eqref{eq:approximate-midpoint-estimate} to
$(V,\lambda)$, letting $T\to\infty$, and using
$\mathbb E_t\widetilde D_T\to0$ and
\eqref{eq:homogeneous-value-representation} yields
\begin{align*}
\widetilde D_t
&\geq
\lim_{T\to\infty}\left[
\frac14\mathbb E_t\int_t^T
\left(
S_s\overline X_s^{t,x-y;w}\cdot\overline X_s^{t,x-y;w}
+|w_s|^2
\right)ds
+\mathbb E_t\widetilde D_T
\right]-\varepsilon\\
&=
\frac14\mathbb E_t\int_t^\infty
\left(
S_s\overline X_s^{t,x-y;w}\cdot\overline X_s^{t,x-y;w}
+|w_s|^2
\right)ds-\varepsilon\\
&\geq
\frac14P_t(x-y)\cdot(x-y)-\varepsilon=\frac12V_t^0(x)+\frac12V_t^0(y)-V_t^0\left(\frac{x+y}{2}\right)-\varepsilon.
\end{align*}

Letting $\varepsilon\downarrow0$ gives $
\frac12G_t(x)+\frac12G_t(y)
-G_t\left(\frac{x+y}{2}\right)\geq0.
$
By continuity of $G$ in $t$ and $x$, the inequality holds
$\mathbb P$-a.s. for all $t\geq0$ and $x,y\in\mathbb R^n$.
Thus $G_t$ is convex and bounded by
\eqref{eq:value-fields-uniform-difference}, hence constant in $x$.
Setting $c_t:=V_t^0(0)-V_t(0)$ defines a bounded continuous
semimartingale satisfying
$
V_t^0(x)=V_t(x)+c_t.
$ Write
$dc_t=\alpha_t\,dt+\sum_{i=1}^d\gamma_t^i\,dW_t^i$.
By uniqueness of the semimartingale decomposition,
$a_t^0(x)-a_t(x)=\alpha_t$ and
$Z_t^{0,i}(x)-Z_t^i(x)=\gamma_t^i$,
$d\mathbb P\otimes dt$-a.e., for each $x$.
By spatial continuity, the null set can be chosen independently of $x$.
Since $c$ is independent of $x$, the spatial derivatives of $V^0$ and
$V$ coincide, as do those of their martingale coefficients.
The It\^o--Wentzell operators therefore satisfy
$\mathcal L_t^uV_t^0(x)=\mathcal L_t^uV_t(x)+\alpha_t$.
Comparing the two stochastic ergodic HJB equations gives
$\lambda_t^0=\lambda_t+\alpha_t$,
$d\mathbb P\otimes dt$-a.e.
Thus $V_t^0(x)-V_t(x)-\int_0^t
(\lambda_s^0-\lambda_s)\,ds$ is independent of $x$
and is a local martingale, as claimed.
\end{proof}

\section{Further characterizations}
\label{sec:further-characterizations}

\subsection{Periodic cases}
\label{subsec:periodic-stationary}
In this subsection, we study random periodic coefficients. Related
ergodic LQ problems with random periodic coefficients were considered
in \cite{WuZhang2026}, although their setting and results differ from
ours.

For this purpose, we impose the following additional structure on the
underlying filtered probability space.

\begin{assumption}
\label{ass:brownian-shift}
The filtered probability space
$(\Omega,\mathcal F,\mathbb P,(\mathcal F_t)_{t\geq0})$
is equipped with a family of measurable maps
$\theta_t:\Omega\to\Omega$, $t\geq0$, satisfying:
\begin{enumerate}[label=\textup{(\roman*)}]
\item Semigroup property:
$\theta_0=\mathrm{Id}$ and
$\theta_{t+s}=\theta_t\circ\theta_s$ for all $s,t\geq0$.

\item Measure-preserving Brownian shifts:
$\mathbb P\circ\theta_t^{-1}=\mathbb P$ for every $t\geq0$, and
$W_s\circ\theta_t=W_{s+t}-W_t$ for every $s,t\geq0$, $\mathbb P$-a.s.
\end{enumerate}
\end{assumption}

\begin{definition}
\label{def:random-periodic}
Fix $\tau>0$. A process $Y$ is called $\tau$-random periodic if
$Y_{t+\tau}=Y_t\circ\theta_\tau$,
$d\mathbb P\otimes dt$-a.e.
A random field $V$ is called $\tau$-random periodic if
$V_{t+\tau}(x)=V_t(x)\circ\theta_\tau$
for every $x\in\mathbb R^n$,
$d\mathbb P\otimes dt$-a.e.
\end{definition}

\begin{remark}
Assumption~\ref{ass:brownian-shift} holds on canonical
Wiener space $\Omega=C_0(\mathbb R_+;\mathbb R^d)$ with Wiener measure,
where $W_t(\omega)=\omega(t)$ and
$
(\theta_t\omega)(s)=\omega(t+s)-\omega(t).
$
This is the standard Brownian shift construction used in random
dynamical systems; see Arnold~\cite{Arnold1998}.
For the notion of random periodic solutions in random dynamical
systems, we refer to Zhao and Zheng~\cite{ZhaoZheng2009}.
The random periodicity relation in
Definition~\ref{def:random-periodic} is also used in
\cite[Definition~2.1]{WuZhang2026}.
\end{remark}

\begin{proposition}
\label{prop:random-periodicity-ergodic-LQ}
Let Assumption~\ref{ass:brownian-shift} hold, and let
$A,B,C,D,f,g,S$ satisfy
Assumptions~\ref{ass:bounded-coefficients} and
\ref{ass:uniform-stabilizability},
and be $\tau$-random periodic for some $\tau>0$.
Then there exists a $\tau$-random periodic solution $(V^0,\lambda^0)$
of the stochastic ergodic HJB equation, and \begin{equation}\label{eq:periodic-ergodic-value}
\lambda_*(K)
=
\lambda^*(K)
=
\frac1{\tau}\int_0^\tau\mathbb E[\lambda_t^0]\,dt,
\qquad K\in\mathcal K.
\end{equation}
\end{proposition}

\begin{proof}
Throughout the proof, all identities between processes are understood
$d\mathbb P\otimes ds$-a.e.

\emph{Step 1: periodicity of the infinite-horizon equations.}
Set $W_s^\tau:=W_{s+\tau}-W_\tau$ and
$\mathcal F_s^\tau:=\mathcal F_{s+\tau}$, $s\geq0$.
Then $W^\tau$ is a Brownian motion with respect to
$\mathbb F^\tau:=(\mathcal F_s^\tau)_{s\geq0}$.
By Assumption~\ref{ass:brownian-shift} and the random periodicity
of the coefficients, both
$(P_{s+\tau},Q_{s+\tau})$ and
$(P_s\circ\theta_\tau,Q_s\circ\theta_\tau)$ solve the same
Riccati equation with generator $\mathcal R_{s+\tau}$,
driven by $W^\tau$ in the filtration $\mathbb F^\tau$.
Both pairs belong to the bounded nonnegative solution class.

Uniqueness for the shifted equation follows from
Proposition~\ref{prop:infinite-horizon-riccati}\textup{(i)}.
Indeed, in the original time variable, any solution in the stated
class can be extended from $[\tau,\infty)$ to $[0,\infty)$
by Proposition~\ref{prop:bsde_finite_horizon},
using its value at $\tau$ as the terminal condition.
The extension belongs to the uniqueness class of
Proposition~\ref{prop:infinite-horizon-riccati}\textup{(i)}.
Consequently,
$
P_{s+\tau}=P_s\circ\theta_\tau,
Q_{s+\tau}=Q_s\circ\theta_\tau.
$

Similarly, the pairs $(r_{s+\tau},k_{s+\tau})$ and
$(r_s\circ\theta_\tau,k_s\circ\theta_\tau)$ solve the same shifted
dual equation, driven by $W^\tau$ in $\mathbb F^\tau$.
The uniqueness argument in the proof of
Proposition~\ref{prop:infinite-horizon-riccati}\textup{(ii)}
also applies on $[\tau,\infty)$.
Hence
$
r_{s+\tau}=r_s\circ\theta_\tau,
k_{s+\tau}=k_s\circ\theta_\tau.
$

\emph{Step 2: periodicity of $(V^0,\lambda^0)$.}
By \eqref{eq:quadratic-ergodic-value-field}, the periodicity of
$(P,Q,r,k)$ and the coefficients implies
$
V_{t+\tau}^0(x)=V_t^0(x)\circ\theta_\tau, \lambda_{t+\tau}^0=\lambda_t^0\circ\theta_\tau.
$

Set $m(t):=\mathbb E[\lambda_t^0]$ and
$\overline\lambda:=\tau^{-1}\int_0^\tau m(t)\,dt$.
Since $\theta_\tau$ preserves $\mathbb P$, $m$ is
$\tau$-periodic almost everywhere.
Therefore,
$\Xi(T):=\int_0^T(m(t)-\overline\lambda)\,dt$ is bounded,
and it then follows that
$\left|\mathbb E\int_0^\infty K_\rho(t)\lambda_t^0\,dt-\overline\lambda\right|
\leq \rho K(0)\|\Xi\|_\infty\longrightarrow0$
as $\rho\downarrow0$, which yields
\eqref{eq:periodic-ergodic-value}.
\end{proof}

\begin{remark}[Stationary coefficients]
The stationary case can be treated similarly in a two-sided
Brownian framework, with processes and filtrations indexed by
$\mathbb R$ rather than $\mathbb R_+$, which we do not study here.
In that setting, our construction yields a stationary canonical
solution $(V^0,\lambda^0)$,
with $\lambda_*(K)=\lambda^*(K)=\mathbb E\int_0^1\lambda_t^0\,dt$
for all $K\in\mathcal K$.
When $g=0$, this characterization of the optimal ergodic cost
is consistent with \cite{GuatteriMasiero2009Stationary},
where the stationary control problem is studied through
stationary solutions of the infinite-horizon Riccati and dual equations.
\end{remark}

\subsection{Stability}
\label{subsec:stability}
Throughout this subsection, let $\Theta=(A,B,C,D,S,f,g)$ and
$\widetilde\Theta=(\widetilde A,\widetilde B,\widetilde C,
\widetilde D,\widetilde S,\widetilde f,\widetilde g)$ be two coefficient
tuples satisfying Assumptions~\ref{ass:bounded-coefficients} and
\ref{ass:uniform-stabilizability}.
Without loss of generality, let $\delta,M>0$ be common constants
in Assumption~\ref{ass:uniform-stabilizability} for both tuples.
For $\Theta$, set
$\|\Theta\|_{\mathbb L^\infty}
:=\max_{\Phi\in\{A,B,C,D,S,f,g\}}
\|\Phi\|_{\mathbb L^\infty}$,
and denote by $(P,Q,r,k)$ the corresponding solution from
Proposition~\ref{prop:infinite-horizon-riccati}, and by $(V^0,\lambda^0)$
the associated solution defined in
\eqref{eq:quadratic-ergodic-value-field}.
We evaluate the coefficient maps in \eqref{eq:LQ-coefficient-maps}
at $(P,Q,r,k)$ and omit their arguments.
For $\widetilde\Theta$, we denote the corresponding coefficients,
solutions, and associated quantities by the same symbols with tildes,
and use the same componentwise norm for
$\|\Theta-\widetilde\Theta\|_{\mathbb L^\infty}$.
We write $\|\cdot\|_\infty$ for the $L^\infty(\Omega,\mathbb P)$ norm
of random variables.

\begin{proposition}
\label{prop:stability-infinite-horizon-equations}
There exists a constant $C>0$, depending only on
$\|\Theta\|_{\mathbb L^\infty}$,
$\|\widetilde\Theta\|_{\mathbb L^\infty}$, $\delta$, and $M$,
such that the following hold.

\hspace{-5mm}\textup{(i)}
$
\|P-\widetilde P\|_{\mathbb S^\infty}
+\|r-\widetilde r\|_{\mathbb S^\infty}
+\|Q-\widetilde Q\|_{\mathbb L^2_{\mathrm{unif}}}
+\|k-\widetilde k\|_{\mathbb L^2_{\mathrm{unif}}}
+\|\lambda^0-\widetilde\lambda^0\|_{\mathbb L^1_{\mathrm{unif}}}
\leq C\|\Theta-\widetilde\Theta\|_{\mathbb L^\infty}.$

\hspace{-5mm}\textup{(ii)}
$
|\lambda_*(K)-\widetilde\lambda_*(K)|,\,
|\lambda^*(K)-\widetilde\lambda^*(K)|
\leq
C\lim_{t\to\infty}
\operatorname*{ess\,sup}_{s\geq t}
\|\Theta_s-\widetilde\Theta_s\|_\infty, K\in\K.
$
\end{proposition}

\begin{proof}
Throughout the proof, $C$ denotes a generic positive constant, which
may change from line to line and depends only on the quantities
specified in the statement, while $C_L$ may additionally depend on
the interval length $L$. 
We first prove \textup{(i)}. Set $\varepsilon:=\|\Theta-\widetilde\Theta\|_{\mathbb L^\infty}$.

\emph{Step 1: stability of $P$.}
Fix $t\geq0$, $L>0$, and $x\in\mathbb R^n$. Let
$(\widetilde X,\widetilde \nu)$ be the optimal homogeneous state-control
pair for the second system starting from $(t,x)$, given by
Lemma~\ref{lem:homogeneous-closed-loop}, and let $X$ be the
state of the first system under the same control $\widetilde \nu$. By \eqref{eq:exponential_stability} and the value representation
\eqref{eq:homogeneous-value-representation},
\[
\mathbb E_t|\widetilde X_{t+L}|^2
\leq Ce^{-\beta L}|x|^2,
\qquad
\mathbb E_t\int_t^\infty
\left(|\widetilde X_s|^2+|\widetilde \nu_s|^2\right)ds
\leq C|x|^2.
\]
Set $\Delta X:=X-\widetilde X$, and similarly for the
coefficients. Then $\Delta X_t=0$ and
\[
\begin{aligned}
d\Delta X_s
={}&
\left(A_s\Delta X_s+\Delta A_s\widetilde X_s
+\Delta B_s\widetilde \nu_s\right)ds
+
\sum_{i=1}^d
\left(C_s^i\Delta X_s+\Delta C_s^i\widetilde X_s
+\Delta D_s^i\widetilde \nu_s\right)dW_s^i.
\end{aligned}
\]
By the BDG and Gronwall inequalities, we see that
\begin{equation}\label{eq:stability-state-difference}
\mathbb E_t\sup_{t\leq s\leq t+L}|X_s-\widetilde X_s|^2
\leq
C_L\varepsilon^2
\mathbb E_t\int_t^{t+L}
\left(|\widetilde X_s|^2+|\widetilde \nu_s|^2\right)ds
\leq C_L\varepsilon^2|x|^2.
\end{equation}
Since $(P,Q)|_{[t,t+L]}$ and
$(\widetilde P,\widetilde Q)|_{[t,t+L]}$ solve the corresponding
finite-horizon backward stochastic Riccati equations with terminal matrices $P_{t+L}$ and
$\widetilde P_{t+L}$, respectively,
Proposition~\ref{prop:bsde_finite_horizon}, applied to the homogeneous
case $f=g=0$, gives
\[
\begin{aligned}
P_tx\cdot x
&\leq
\mathbb E_t\int_t^{t+L}
\left(S_sX_s\cdot X_s+|\widetilde\nu_s|^2\right)ds
+\mathbb E_t\left[P_{t+L}X_{t+L}\cdot X_{t+L}\right],\\
\widetilde P_tx\cdot x
&=
\mathbb E_t\int_t^{t+L}
\left(\widetilde S_s\widetilde X_s\cdot\widetilde X_s
+|\widetilde\nu_s|^2\right)ds
+\mathbb E_t
\left[\widetilde P_{t+L}\widetilde X_{t+L}\cdot
\widetilde X_{t+L}\right].
\end{aligned}
\]
Expanding the difference of the quadratic terms and using
\eqref{eq:stability-state-difference} and the
Cauchy--Schwarz inequality, we obtain
\[
\begin{aligned}
\left|
\mathbb E_t\int_t^{t+L}
\left(S_sX_s\cdot X_s
-\widetilde S_s\widetilde X_s\cdot\widetilde X_s\right)ds
\right|
&\leq C_L\varepsilon|x|^2,\\[1em]
\mathbb E_t\left[
P_{t+L}X_{t+L}\cdot X_{t+L}
-\widetilde P_{t+L}\widetilde X_{t+L}\cdot\widetilde X_{t+L}
\right]
&\leq
\left(
Ce^{-\beta L}\|P_{t+L}-\widetilde P_{t+L}\|_\infty
+C_L\varepsilon
\right)|x|^2.
\end{aligned}
\]
Consequently,
$
(P_t-\widetilde P_t)x\cdot x
\leq
\left(
Ce^{-\beta L}\|P_{t+L}-\widetilde P_{t+L}\|_\infty
+C_L\varepsilon
\right)|x|^2.
$ Interchanging the roles of $P$ and $\widetilde P$ and using
the arbitrariness of $x\in\mathbb R^n$, we obtain
\[
\|P_t-\widetilde P_t\|_\infty
\leq
Ce^{-\beta L}\|P_{t+L}-\widetilde P_{t+L}\|_\infty
+C_L\varepsilon.
\]
Choose $L$ such that $q:=Ce^{-\beta L}<1$. Taking the supremum over
$t\geq0$ yields
$
\|P-\widetilde P\|_{\mathbb S^\infty(\mathcal S^n)}
\leq
q\|P-\widetilde P\|_{\mathbb S^\infty(\mathcal S^n)}
+C_L\varepsilon,
$
and hence
$
\|P-\widetilde P\|_{\mathbb S^\infty(\mathcal S^n)}
\leq C\varepsilon.
$

\emph{Step 2: stability of $Q$.}
Set $\Delta P:=P-\widetilde P$ and
$\Delta Q:=Q-\widetilde Q$. Using
$N_s,\widetilde N_s\geq I_k$, the uniform boundedness of
$P$ and $\widetilde P$, and the definitions of $N$, $L$, and the
Riccati generator,
\[
|\Delta\mathcal R_s|
\leq
C\left[
(\varepsilon+|\Delta P_s|)
(1+|Q_s|^2+|\widetilde Q_s|^2)
+
(1+|Q_s|+|\widetilde Q_s|)|\Delta Q_s|
\right].
\]
Applying It\^o's formula to $|\Delta P|^2$ on $[t,t+1]$,
after localization, gives the following estimates, where we use
$\|\Delta P\|_{\mathbb S^\infty}\leq C\varepsilon$
and Young's inequality:
\begin{align*}
\mathbb E_t\int_t^{t+1}|\Delta Q_s|^2\,ds
&\leq \mathbb E_t|\Delta P_{t+1}|^2
+2\mathbb E_t\int_t^{t+1}
|\Delta P_s||\Delta\mathcal R_s|\,ds\\
&\leq C\varepsilon^2
+\frac12\mathbb E_t\int_t^{t+1}|\Delta Q_s|^2\,ds
+C\varepsilon^2\mathbb E_t\int_t^{t+1}
(1+|Q_s|^2+|\widetilde Q_s|^2)\,ds.
\end{align*}
As shown in the proof of
Proposition~\ref{prop:infinite-horizon-riccati}\textup{(i)},
the $\mathbb L^2_{\mathrm{unif}}$ norms of $Q$ and $\widetilde Q$
are bounded by a constant depending only on
$\|\Theta\|_{\mathbb L^\infty}$,
$\|\widetilde\Theta\|_{\mathbb L^\infty}$, and $M$.
Using these bounds and absorbing the $\Delta Q$ term, we obtain
$\|\mathbb E_t\int_t^{t+1}|\Delta Q_s|^2\,ds\|_\infty
\leq C\varepsilon^2$.
Therefore,
$\|Q-\widetilde Q\|_{\mathbb L^2_{\mathrm{unif}}((\mathcal S^n)^d)}
\leq C\varepsilon$.

\emph{Step 3: stability of $r$.} Set $\Delta\Gamma_t:=\Gamma_t-\widetilde\Gamma_t$,
and define $\Delta H$, $\Delta M$, and $\Delta\varphi$ similarly.
Set $\Delta r:=r-\widetilde r,\Delta k:=k-\widetilde k$ and $\zeta_s:=(\Delta H_s)^\top\widetilde r_s
+\sum_{i=1}^d\!(\Delta M_s^i)^\top\widetilde k_s^i
+\Delta\varphi_s$, we have
\[
d\Delta r_s
=-\bigl(
H_s^\top\Delta r_s
+\sum_{i=1}^d\!(M_s^i)^\top\Delta k_s^i
+\zeta_s
\bigr)ds
+\sum_{i=1}^d\!\Delta k_s^i\,dW_s^i.
\]
The estimates in Steps~1 and~2 give
\[
\|\Delta\Gamma\|_{\mathbb L^2_{\mathrm{unif}}}
+\|\Delta H\|_{\mathbb L^2_{\mathrm{unif}}}
+\|\Delta M\|_{\mathbb L^2_{\mathrm{unif}}}
+\|\Delta\varphi\|_{\mathbb L^2_{\mathrm{unif}}}
\leq C\varepsilon.
\]
Fix $t\geq0$ and $x\in\mathbb R^n$, and let
$\overline X^{t,x,*}$ be the optimal homogeneous closed-loop state
given by Lemma~\ref{lem:homogeneous-closed-loop}.
Applying It\^o's formula to
$\Delta r_s\cdot\overline X_s^{t,x,*}$ gives, for every $T>t$,
\[
\Delta r_t\cdot x
=
\mathbb E_t\left[\Delta r_T\cdot\overline X_T^{t,x,*}\right]
+
\mathbb E_t\int_t^T
\zeta_s\cdot\overline X_s^{t,x,*}\,ds.
\]
For every $Y\in\mathbb L^2_{\mathrm{unif}}$, conditional Fubini's theorem
and the tower property give
\[
\mathbb E_t\left(\int_t^{t+1}|Y_s|^2ds\right)^2
=
2\mathbb E_t\int_t^{t+1}|Y_s|^2\mathbb E_s\int_s^{t+1}|Y_r|^2dr\,ds
\leq
2\|Y\|_{\mathbb L^2_{\mathrm{unif}}}^4.
\]
Let $I_n:=[t+n,t+n+1]$. By
\eqref{eq:exponential_stability} and
\eqref{eq:homogeneous-value-representation},
\[
\mathbb E_t\int_{I_n}
\left(
|\overline X_s^{t,x,*}|^2+|\overline\nu_s^{t,x,*}|^2
\right)ds
\leq
C\mathbb E_t
\left[
P_{t+n}\overline X_{t+n}^{t,x,*}\cdot
\overline X_{t+n}^{t,x,*}
\right]
\leq
Ce^{-\beta n}|x|^2.
\]
Since the coefficients of the original state equation are bounded, the
conditional BDG inequality gives
$
\mathbb E_t\sup_{s\in I_n}
|\overline X_s^{t,x,*}|^2
\leq Ce^{-\beta n}|x|^2.
$

The contributions of $(\Delta H_s)^\top\widetilde r_s$
and $\Delta\varphi_s$ in $\zeta_s$ are estimated by
Cauchy--Schwarz, using the boundedness of
$\widetilde r$. For $(\Delta M_s^i)^\top\widetilde k_s^i$,  H\"older's
inequality and the preceding Fubini estimate give
\[
\begin{aligned}
&\mathbb E_t\int_{I_n}|\Delta M_s|\,|\widetilde k_s|\,|\overline X_s^{t,x,*}|\,ds
\\\leq&
\Big(\mathbb E_t\sup_{s\in I_n}|\overline X_s^{t,x,*}|^2\Big)^{\frac12}
\Big[\mathbb E_t\left(\int_{I_n}\!\!|\Delta M_s|^2\,ds\right)^2\Big]^{\frac{1}{4}}
\Big[\mathbb E_t\left(\int_{I_n}\!\!|\widetilde k_s|^2\,ds\right)^2\Big]^{\frac14}
\leq C\varepsilon e^{-\beta n/2}|x|.
\end{aligned}
\]
Consequently,
$
\mathbb E_t\int_{I_n}
\left|\zeta_s\cdot\overline X_s^{t,x,*}\right|ds
\leq C\varepsilon e^{-\beta n/2}|x|.
$
Summing over $n\geq0$ shows that the integral is absolutely convergent.
The terminal term tends to zero by the boundedness of $\Delta r$ and
\eqref{eq:exponential_stability}. Letting $T\to\infty$ therefore gives
$|\Delta r_t\cdot x|\leq C\varepsilon|x|$,
and hence
$\|\Delta r\|_{\mathbb S^\infty(\mathbb R^n)}\leq C\varepsilon$.

\emph{Step 4: stability of $k$.}
By Cauchy--Schwarz and the preceding estimate, we have
$\|\zeta\|_{\mathbb L^1_{\mathrm{unif}}(\mathbb R^n)}
\leq C\varepsilon$.
Applying It\^o's formula to $|\Delta r|^2$ on $[t,t+1]$ and using
Young's inequality, \eqref{eq:infinite-riccati-Q-uniform-bound}, and \eqref{eq:dual-coefficient-bound} 
yields
\[
\mathbb E_t\int_t^{t+1}|\Delta k_s|^2ds
\leq
C\|\Delta r\|_{\mathbb S^\infty}^2
+
C\|\Delta r\|_{\mathbb S^\infty}
\mathbb E_t\int_t^{t+1}|\zeta_s|ds.
\]
Using the estimate for $\Delta r$ gives
$
\|k-\widetilde k\|_{\mathbb L^2_{\mathrm{unif}}((\mathbb R^n)^d)}
\leq C\varepsilon.
$

\emph{Step 5: stability of $\lambda^0$.}
The preceding estimates and definitions of
$\eta,\widetilde\eta,N,\widetilde N$ give
\[
\|\eta-\widetilde\eta\|_{\mathbb L^2_{\mathrm{unif}}}
\leq C\varepsilon,
\qquad
\|\eta\|_{\mathbb L^2_{\mathrm{unif}}}
+\|\widetilde\eta\|_{\mathbb L^2_{\mathrm{unif}}}
\leq C,
\qquad
\|N^{-1}-\widetilde N^{-1}\|_{\mathbb L^\infty}
\leq C\varepsilon.
\]
Since $N,\widetilde N\geq I_k$,
\[
\left|
N_s^{-1}\eta_s\cdot\eta_s
-\widetilde N_s^{-1}\widetilde\eta_s\cdot\widetilde\eta_s
\right|
\leq
|\eta_s-\widetilde\eta_s|
\left(|\eta_s|+|\widetilde\eta_s|\right)
+|N_s^{-1}-\widetilde N_s^{-1}|\,|\widetilde\eta_s|^2.
\]
Using \eqref{eq:quadratic-ergodic-value-field},
Cauchy--Schwarz, and step 3, we obtain
$\|\lambda^0-\widetilde\lambda^0\|_{\mathbb L^1_{\mathrm{unif}}}
\leq C\varepsilon$.

To prove \textup{(ii)}, set
$d(t):=\operatorname*{ess\,sup}_{s\geq t}
\|\Theta_s-\widetilde\Theta_s\|_\infty$.
The proof of \textup{(i)} applies on $[t,\infty)$
with $\varepsilon$ replaced by $d(t)$, and therefore
$
\big\|
\mathbb E_t\int_t^{t+1}
|\lambda_s^0-\widetilde\lambda_s^0|\,ds
\big\|_\infty
\leq Cd(t)$ for $t\geq0.
$
For every $K\in\mathcal K$, the monotonicity of $K_\rho$ and $d$ gives

\begin{align}\label{eq:new}
|\Lambda^{K_\rho}-\widetilde\Lambda^{K_\rho}|
&\leq
\sum_{j=0}^\infty K_\rho(j)
\mathbb E\int_j^{j+1}
|\lambda_s^0-\widetilde\lambda_s^0|\,ds\nonumber\\
&\leq C\sum_{j=0}^\infty K_\rho(j)d(j)
\leq C\rho K(0)d(0)
+C\int_0^\infty \rho K(\rho t)d(t)\,dt.
\end{align}
Since $d$ is bounded, nonincreasing and
$\int_0^\infty K(s)\,ds=1$, dominated convergence yields
\[
\int_0^\infty \rho K(\rho t)d(t)\,dt
=
\int_0^\infty K(s)d(s/\rho)\,ds
\xrightarrow[\rho\downarrow0]{}
\lim_{t\to\infty}d(t),
\]
and we obtain \textup{(ii)} by sending  $\rho\downarrow0$ in~\eqref{eq:new}.
\end{proof}

\appendix

\section{Averaging kernels}
\label{app:averaging-kernels}
We establish here the analytic properties of the averaging
kernels used in the verification argument. We first derive their
mixture representation in terms of normalized finite-horizon averages
and then prove the Tauberian result relating the Ces\`aro, exponential,
and general kernel averages.

\begin{lemma}\label{lemma:upper-sandwich}
Let $K\!\in\!\K$, and $f\!\in\! L^1_{\mathrm{loc}}(\R_+)$ with
$\sup_{t\geq0}\int_t^{t+1}\!\!f^-(s)\,ds<\infty$. Then
\[
\limsup_{\rho\downarrow0}\int_0^\infty K_{\rho}(t)f(t)\,dt
\leq
\limsup_{T\to\infty}\frac{1}{T}\int_0^T f(t)\,dt.
\]
\end{lemma}

\begin{proof}
Set $F(T):=\int_0^T f(t)\,dt$, $A(T):=F(T)/T$, and
$L:=\limsup_{T\to\infty}A(T)$. The assumption on $f^-$ implies
$F(T)\geq-C(1+T)$, so $L>-\infty$. If $L=+\infty$, there is nothing to
prove. Otherwise, for $\varepsilon>0$, there exists $T_\varepsilon>0$
such that $A(T)\leq L+\varepsilon$ for all $T\geq T_\varepsilon$.

Since $K\in\K$, the measure $\mu:=-dK$ is a positive Stieltjes measure on
$\R_+$ and
$\int_0^\infty s\,\mu(ds)=\int_0^\infty K(t)\,dt=1$. Thus
$\nu(ds):=s\,\mu(ds)$ is a probability measure. Moreover,
$L<\infty$ and the lower bound on $F$ imply $|F(T)|\leq C(1+T)$ for
large $T$, while the monotonicity and integrability of $K$ imply
$TK(T)\to0$. Hence the boundary term at infinity vanishes, and by
Lebesgue--Stieltjes integration by parts,
\[
\int_0^\infty K_{\rho}(t)f(t)\,dt
=
\int_0^\infty A\!\left(\frac{s}{\rho}\right)\nu(ds).
\]
Splitting the last integral at $\rho T_\varepsilon$, we obtain
\[
\int_0^\infty K_{\rho}(t)f(t)\,dt
\leq
\rho\sup_{0\leq T\leq T_\varepsilon}F(T)^+
\mu\bigl((0,\rho T_\varepsilon]\bigr)
+
(L+\varepsilon)\nu\bigl((\rho T_\varepsilon,\infty)\bigr).
\]
The first term is bounded by
$\rho K(0)\sup_{0\leq T\leq T_\varepsilon}F(T)^+$ and tends to zero.
Since $\nu(\{0\})=0$,
$\nu\bigl((\rho T_\varepsilon,\infty)\bigr)\to1$ as $\rho\downarrow0$.
Therefore,
$
\limsup_{\rho\downarrow0}\int_0^\infty K_{\rho}(t)f(t)\,dt
\leq L+\varepsilon.
$
Letting $\varepsilon\downarrow0$ yields the result.
\end{proof}

\begin{proof}[Proof of Proposition~\ref{prop:relation-value}]
Set $f(t):=\mathbb E[\lambda_t]$, then $
\sup_{t\geq0}\int_t^{t+1}|f(s)|\,ds<\infty,$ since
$\lambda\in\mathbb L^1_{\mathrm{unif}}.$
Applying Lemma~\ref{lemma:upper-sandwich} to $\pm f$ yields
$
\lambda_*(K^0)
\leq
\lambda_*(K)
\leq
\lambda^*(K)
\leq
\lambda^*(K^0),
$ for $K\in\K$.
Hence equality for $K^0$ implies equality for every $K\in\K$. Let $\widehat K\in\mathcal K_W$ and assume
$\lambda_*(\widehat K)=\lambda^*(\widehat K)$.
Replacing $\lambda$ by $\lambda\mathbf1_{(1,\infty)}$
does not change $\lambda_*(K)$ or $\lambda^*(K)$
for any $K\in\mathcal K$.
Thus, we may assume that $\lambda=0$ on $[0,1]$,
and hence $f=0$ on $[0,1]$.
Set $F(T):=\int_0^T f(t)\,dt$ and $s(T):=F(T)/T$.
The uniform local integrability of $f$ then implies that
$s$ is bounded on $(0,\infty)$ and, for $T\geq2$ and $\rho$ close to $1$,
$
|s(T)-s(\rho T)|
\leq C\left(|1-\rho|+\frac1T\right).
$
Thus $s$ satisfies the slow oscillation condition in
Korevaar~\cite[Chapter~II, equation~(2.8)]{Korevaar2004}. Set $k(t):=-t\widehat K'(t)$. Since $\widehat K\in\mathcal K_W$,
$k\geq0$, while
$
0\leq t\widehat K(t)\leq2\int_{t/2}^t\widehat K(r)\,dr\to0
$
and integration by parts gives
$
\int_0^\infty k(t)\,dt=1.
$
Thus, by Theorems~3.4 and~8.1
in~\cite[Chapter~II]{Korevaar2004}
and the preceding inequalities,
$\lambda_*(K)=\lambda^*(K)$ for every $K\in\mathcal K$.
\end{proof}
\section{Approximate feedback selection}
\label{app:exponential-map}
We construct here a progressively measurable feedback map
that is locally bounded and locally Lipschitz
and satisfies the conditional integrated estimate
\eqref{eq:approximate-feedback-selection}.
These regularity properties ensure the existence of a unique maximal
strong solution to the corresponding closed-loop equation.
Step~1 provides smooth local approximate minimizers.
Step~2 uses convexity and a partition of unity to construct a global
approximate minimizer, while the truncation ensures
finite, $\mathcal F_t$-measurable local bounds and Lipschitz constants.

\begin{lemma}
\label{lem:blockwise-integrated-lipschitz-selector}
Let $G_t(x,u):=u^\top M_t(x)u+2b_t(x)\cdot u+c_t(x)$,
where $M,b,c$ are random fields on
$\Omega\times\mathbb R_+\times\mathbb R^n$ taking values in
$\mathcal S^k$, $\mathbb R^k$, and $\mathbb R$, respectively,
progressively measurable for every $x$,
and continuous in $x$, $d\mathbb P\otimes dt$-a.e.
Suppose that $d\mathbb P\otimes dt$-a.e.
\[
\inf_{u\in\mathbb R^k}G_t(x,u)=0,
\text{for all }x\in\mathbb R^n,
\text{and }
\mathbb E\int_0^T\sup_{|x|\leq R}c_t(x)\,dt<\infty,
\text{for all }T,R>0.
\]
Then, for every $t\geq0$ and $\varepsilon>0$,
there exists a map
$u^{\varepsilon,t}:
\Omega\times\mathbb R_+\times\mathbb R^n\to\mathbb R^k$,
progressively measurable in $(\omega,s)$ for every $x\in\mathbb R^n$,
such that, for every $T,R>0$, there exist finite
$\mathcal F_t$-measurable random variables
$C_{T,R}^t,L_{T,R}^t$ satisfying
\[
|u_s^{\varepsilon,t}(x)|\leq C_{T,R}^t,
|u_s^{\varepsilon,t}(x)-u_s^{\varepsilon,t}(y)|
\leq L_{T,R}^t|x-y|,
\text{for }t\leq s\leq T+t,\ |x|,|y|\leq R, \mathbb P\text{-a.s.}
\]
Moreover, for all $Y$ progressively measurable,
$
\mathbb E_t\int_t^\infty
G_s\bigl(Y_s,u_s^{\varepsilon,t}(Y_s)\bigr)\,ds
\leq\varepsilon,~\mathbb P\text{-a.s.}
$
\end{lemma}

\begin{proof}
Fix $\varepsilon>0$. We prove the result for $t=0$.
For general $t\geq0$, shift the construction by $t$, replace
$\mathbb E$ with $\mathbb E_t$,
and set the feedback equal to zero before $t$. After modifying $G$ on a progressively measurable
$d\mathbb P\otimes ds$-null set, we may assume that
the pointwise assumptions hold everywhere.

\emph{Step 1:}
Choose a countable family
$(v_j)_{j\geq1}\subset C_c^\infty(\mathbb R^n;\mathbb R^k)$
that is dense in $C(\mathbb R^n;\mathbb R^k)$
for uniform convergence on compact sets.
In this step, we show that for every nonempty compact set
$A\subset\mathbb R^n$ and every $(\omega,s)$,
\[
\inf_{j\geq1}\sup_{x\in A}G_s(\omega,x,v_j(x))=0.
\]
Fix $(\omega,s)$ and omit this dependence from the notation for $M,b,c$.
The condition $\inf_{u\in\mathbb R^k}G_s(\omega,x,u)=0$
implies that $M(x)\geq0$ for every $x\in\mathbb R^n$.
For $\rho>0$, define
\[
h_\rho(x):
=G_s(\omega,x,w_\rho(x))+\rho|w_\rho(x)|^2,
\,\text{with }\;
w_\rho(x):=-(M(x)+\rho I_k)^{-1}b(x).
\]
Then $h_\rho$ is continuous and decreases pointwise to zero as
$\rho\downarrow0$. By Dini's theorem,
$
\sup_{x\in A}G_s(\omega,x,w_\rho(x))
\leq\sup_{x\in A}h_\rho(x)\longrightarrow0.
$
By the density of $(v_j)$ and joint continuity of
$G_s(\omega,\cdot,\cdot)$, approximating $w_\rho$ uniformly on $A$
proves the claim.

\emph{Step 2:}
Choose a nonnegative function $\psi\in C_c^\infty(\mathbb R^n)$
that is strictly positive on $[-1,1]^n$, enumerate
$\mathbb Z^n$ as $(z_m)_{m\geq1}$, and define
$
\chi_m(x):=
\frac{\psi(x-z_m)}{\sum_{\ell\geq1}\psi(x-z_\ell)},
A_m:=\operatorname{supp}\chi_m.
$ Then
$\chi_m\in C_c^\infty(\mathbb R^n)$,
$\chi_m\geq0$, and $\sum_{m\geq1}\chi_m(x)=1$ for every
$x\in\mathbb R^n$.
Moreover, every compact subset of $\mathbb R^n$ intersects only
finitely many $A_m$.

Set $I_n:=[n,n+1)$ and
$\delta_n:=\varepsilon2^{-n-2}$, $n\geq0$.
Using the family $(v_j)$ chosen in Step~1, for $s\in I_n$, define
\[
J_{n,m}(s)
:=\min\left\{j\geq1:
\sup_{x\in A_m}G_s(x,v_j(x))<\delta_n\right\},
\quad
H_m(s):=\sup_{x\in A_m}c_s(x).
\]
By Step~1, $J_{n,m}$ is finite.
Both $J_{n,m}$ and $H_m$ are progressively measurable.
Moreover, $\mathbb E\int_{I_n}H_m(s)\,ds<\infty$. By dominated convergence,
$\mathbb E\int_{I_n}H_m(s)
\mathbf1_{\{J_{n,m}(s)>q\}}\,ds\longrightarrow0$
as $q\to\infty$. Therefore,

\[
\mathbb E\int_{I_n}H_m(s)
\mathbf1_{\{J_{n,m}(s)>K_{n,m}\}}\,ds
\leq\varepsilon2^{-n-m-2},
\quad\text{for some deterministic }K_{n,m}\geq1.
\]

Define
\[
u_s^{\varepsilon,0}(x)
:=\sum_{m\geq1}\chi_m(x)
v_{J_{n,m}(s)}(x)
\mathbf1_{\{J_{n,m}(s)\leq K_{n,m}\}},
\qquad s\in I_n.
\]
The sum is locally finite, so $u^{\varepsilon,0}$ is progressively
measurable for every $x$.
For every $T,R>0$, only finitely many indices $n,m$ contribute
on $[0,T]\times\overline B_R$, and the truncation restricts $j$
to a finite set for each pair $(n,m)$.
Since each $\chi_m v_j$ is smooth, the required local boundedness
and Lipschitz estimates hold with finite deterministic constants
$C_{T,R}^0$ and $L_{T,R}^0$. Finally, the convexity of $G$ in $u$ gives, for $s\in I_n$,
\[
G_s(x,u_s^{\varepsilon,0}(x))
\!\leq\!\delta_n+
\sum_{m\geq1}\chi_m(x)c_s(x)
\mathbf1_{\{J_{n,m}(s)>K_{n,m}\}}
\!\leq\!\delta_n+\!
\sum_{m\geq1}H_m(s)
\mathbf1_{\{J_{n,m}(s)>K_{n,m}\}}.
\]
Consequently, for every progressively measurable process $Y$,
\[
0\leq
\mathbb E\int_0^\infty
G_s(Y_s,u_s^{\varepsilon,0}(Y_s))\,ds
\leq
\sum_{n\geq0}\delta_n
+\sum_{n\geq0}\sum_{m\geq1}\varepsilon2^{-n-m-2}
=\varepsilon.
\]
\end{proof}

\vspace{-7mm}
\bibliographystyle{siamplain}
\bibliography{mybibliography}

\end{document}